\documentclass{amsart}
\usepackage{amssymb,amsfonts,array}
\usepackage{amsmath}
\usepackage[utf8]{inputenc}
\usepackage{mathtools}
\usepackage{tabulary}
\usepackage[all,arc]{xy}
\usepackage{enumerate}
\usepackage{easyReview}
\usepackage{mathrsfs}
\usepackage{xcolor}

\usepackage[hidelinks]{hyperref}

\usepackage{bm}
\usepackage{dsfont}
\usepackage{pgfplots}
\usepackage{tikz}
\usepackage{multirow}
\usepackage{tikz-cd}
\usetikzlibrary{cd,shapes,arrows,positioning,calc}
\usepackage{tikz}
\usetikzlibrary{intersections,through,backgrounds}
\usepackage{tkz-base}
\usepackage{tkz-euclide}
\usepackage{geometry}
\usepackage[OT2,T1]{fontenc}
\DeclareSymbolFont{cyrletters}{OT2}{wncyr}{m}{n}
\DeclareMathSymbol{\Sha}{\mathalpha}{cyrletters}{"58}

\newtheorem{thm}{Theorem}[section]
\newtheorem{cor}[thm]{Corollary}
\newtheorem{con}[thm]{Conjecture}
\newtheorem{prop}[thm]{Proposition}
\newtheorem{lemma}[thm]{Lemma}

\theoremstyle{definition}
\newtheorem{defn}[thm]{Definition}

\newtheorem{rmk}[thm]{Remark}
\newtheorem{assump}[thm]{Assumption}

\numberwithin{equation}{section}

\newcommand{\floor}[1]{\lfloor#1\rfloor}

\newcommand{\FF}{\mathbb{F}}
\newcommand{\DD}{\mathbb{D}}
\newcommand{\LL}{\mathbb{L}}
\newcommand{\XX}{\mathbb{X}}

\newcommand{\GG}{\mathbb{G}}
\newcommand{\BB}{\mathbb{B}}
\newcommand{\QQ}{\mathbb{Q}}
\newcommand{\ZZ}{\mathbb{Z}}

\newcommand{\cO}{\mathcal{O}}
\newcommand{\cU}{\mathcal{U}}

\newcommand{\cH}{\mathcal{H}}

\newcommand{\cF}{\mathcal{F}}

\newcommand{\bC}{\mathbf{C}}
\newcommand{\bM}{\mathbf{M}}

\newcommand{\bH}{\mathbf{H}}

\newcommand{\Alg}{\mathrm{Alg}}
\newcommand{\Ch}{\mathrm{Ch}}
\newcommand{\can}{\mathrm{can}}
\newcommand{\Tr}{\mathrm{Tr}}

\newcommand{\Til}{\mathrm{Til}}
\newcommand{\Av}{\mathrm{Av}}
\newcommand{\Rep}{\mathrm{Rep}}

\newcommand{\coh}{\mathrm{coh}}
\newcommand{\End}{\mathrm{End}}

\newcommand{\Mod}{\mathrm{Mod}}

\newcommand{\cind}{{\mathrm{c}\mbox{-}\mathrm{Ind}}}
\newcommand{\cInd}{{\mathrm{c}\mbox{-}\mathrm{Ind}}}
\newcommand{\Ind}{{\mathrm{Ind}}}

\newcommand{\Isoc}{\mathrm{Isoc}}
\newcommand{\Shv}{\mathrm{Shv}}
\newcommand{\fg}{\mathrm{f.g.}}
\newcommand{\GS}{\mathrm{GS}}
\newcommand{\Iw}{\mathrm{Iw}}
\newcommand{\tame}{\mathrm{tame}}
\newcommand{\unip}{\mathrm{unip}}
\newcommand{\op}{\mathrm{op}}
\newcommand{\Perf}{{\mathrm{Perf}}}
\newcommand{\QCoh}{\mathrm{QCoh}}
\newcommand{\mon}{\mathrm{mon}}

\newcommand{\cl}{\mathrm{cl}}
\newcommand{\Morita}{\mathrm{Morita}}
\newcommand{\PR}{\mathrm{Pr}}
\newcommand{\pr}{\mathrm{pr}}
\newcommand{\CohSpr}{\mathrm{CohSpr}}
\newcommand{\Hom}{\mathrm{Hom}}

\newcommand{\LS}{\mathrm{LS}}
\newcommand{\cZ}{\mathcal{Z}}

\newcommand{\GL}{\mathrm{GL}}
\newcommand{\git}{{/\!\!/}}
\newcommand{\Coh}{\mathrm{Coh}}
\newcommand{\Ad}{\mathrm{Ad}}
\newcommand{\Comm}{\mathrm{Comm}}
\newcommand{\lie}{\operatorname{Lie}}
\newcommand{\ghat}{\hat{G}}
\newcommand{\that}{\hat{T}}

\newcommand{\fq}{\mathbb{F}_q}

\newcommand{\flbar}{\overline{\mathbb{F}}_{\ell}}

\newcommand{\rep}{\operatorname{Rep}}

\newcommand{\Spec}{\operatorname{Spec}}

\title{On the normality of the commuting scheme}
\author[Jack Sempliner]{Jack Sempliner}
\address{Mathematical Sciences Building, 520 Portola Plaza, Los Angeles, CA 90095, United States of America}
\email{jsempliner@math.ucla.edu}
\author[Xiangqian Yang]{Xiangqian Yang}
\address{Beijing International Center for Mathematical Research, Peking University, Beijing 100871, China}
\email{yangxq@pku.edu.cn}

\begin{document}
\begin{abstract}
    If $G$ is a connected reductive algebraic group over a field $k$ of characteristic $\ell$, we show that both the Lie-theoretic and group-theoretic commuting schemes associated to $G$ are normal and Cohen--Macaulay when $\ell = 0$, or $\ell$ is larger than the Coxeter numbers of all simple factors of $G$, and that $\ell \neq 19$ (resp. $\ell \neq 31$) when $G$ has a simple factor of type $E_7$ (resp. $E_8$). The main new ingredients are provided by studying the unipotent categorical Langlands functor of Zhu. Along the way we prove a $t$-exactness result for a finite unipotent categorical Langlands functor which may be of independent interest.
\end{abstract}

\maketitle

\tableofcontents

\section{Introduction}

Let $k$ be a field and let $\mathfrak{g}/k$ be a reductive Lie algebra, and let $[-, -]: \mathfrak{g}\times \mathfrak{g} \to \mathfrak{g}$ denote the commutator map. The commuting scheme $\Comm_{\mathfrak{g}}$ is defined to be the scheme-theoretic fiber of this map above $0 \in \mathfrak{g}$. In 1955 in the work of Motzkin--Taussky \cite{MotzkinTaussky1955}, it was shown that when $\mathfrak{g} = \mathfrak{gl}_n$ the scheme $\Comm_{\mathfrak{g}}$ is irreducible. In 1979, Richardson \cite{Richardson1979} showed that $\Comm_{\mathfrak{g}}$ is always irreducible over a field $k$ of characteristic zero. This result was generalized by work of Levy \cite{Levy} to fields $k$ such that $\ell = \mathrm{char}(k)$ is \emph{good} for $\mathfrak{g}$ under minor additional technical hypotheses. It remains quite a natural question whether the obvious equations suffice to define the commuting variety $\Comm_{\mathfrak{g}, \mathrm{red}}$, i.e. whether the scheme $\Comm_{\mathfrak{g}}$ is reduced. Indeed, the following folklore conjecture seems to have been proposed in 1982 by M. Artin and M. Hochster in the case $\mathfrak{g} = \mathfrak{gl}_n$. We have taken the liberty of stating an extension of this conjecture to any $\mathfrak{g}$. 

\begin{con}\label{con: main conjecture}
    If the prime $\ell = \mathrm{char}(k)$ is pretty good\footnote{By this we mean that there exists a reductive algebraic group $G/k$ such that $\lie(G) = \mathfrak{g}$ and $\ell$ is pretty good for $G$. We refer to \cite[Definition 2.11]{Herpel2013} for the definition of "pretty good". We note that "pretty good" is stronger than "good" but weaker than "very good".} for $\mathfrak{g}$, then the scheme $\Comm_{\mathfrak{g}}$ is normal and Cohen--Macaulay.
\end{con}

In the case that $\mathfrak{g} = \mathfrak{gl}_n$ this conjecture was resolved positively in the case $n = 3$ in \cite{Thompson1985} and for $n = 4$ in \cite{Hreinsdottir1994}. Work of Ginzburg in \cite{Ginzburg2012} showed that in characteristic zero the normalization $\widetilde{\Comm}_{\mathfrak{g}}$ is reduced and Cohen--Macaulay for any reductive $\mathfrak{g}$. Work of Li--Nadler--Yun in \cite{LNY24} demonstrated that the ring of $G$-invariant functions $\cO(\Comm_G)^{G}$ is reduced in characteristic zero. Nonetheless, to the best of the authors' knowledge, this conjecture remains unresolved even when $k$ has characteristic zero and $\mathfrak{g} = \mathfrak{gl}_n$ for $n > 4$. 

Given a connected reductive algebraic group $G/k$, one may define analogously the \emph{group-theoretic} commuting scheme
\[
    \Comm_G=\{(x,y) \in G \times G\,|\,xyx^{-1}=y\}.
\]
The main result of this paper is the following theorem, which in particular resolves much of Conjecture \ref{con: main conjecture}. 

\begin{thm}[{Corollaries~\ref{cor: char 0}, \ref{cor: char 0 Lie}}]\label{thm: intro main theorem}
    Let $k$ be a field of characteristic $\ell$ and let $G/k$ be a connected reductive algebraic group, and let $\mathfrak{g}$ denote its Lie algebra. If $\ell \neq 0$, suppose that $\ell$ is larger than the Coxeter numbers of the simple factors of $G$, and that $\ell \neq 19$ (resp. $\ell \neq 31$) when $G$ has a simple factor of type $E_7$ (resp. $E_8$). Then the schemes $\Comm_G, \Comm_{\mathfrak{g}}$ are normal and Cohen--Macaulay. 
\end{thm}

Perhaps surprisingly, the proof of Theorem \ref{thm: intro main theorem} does not follow from a detailed analysis of the particular geometry of the schemes in question. Instead, the main input arises from an ongoing program of the authors to understand the tame categorical Langlands correspondence of Zhu \cite{ZTame} in non-banal characteristics.

From now on we denote by $\hat{G}$ the connected reductive group in Theorem \ref{thm: intro main theorem}, as it will appear on the spectral side of the categorical Langlands correspondence. We briefly explain the proof of Cohen--Macaulayness of $\Comm_{\hat{G}}$ in Theorem \ref{thm: intro main theorem}. We assume that $\hat{G}$ has simply connected derived subgroup for simplicity. We can easily reduce to the case $k=\overline{\FF}_\ell$. Let $\hat{C}=\hat{G}\git\hat{G}$ denote the Chevalley quotient. We have a natural map
\[\pi\colon \Comm_{\hat{G}}\to \hat{C}\]
by sending a commuting pair $(x,y)$ to the image of $y$. Let $u\in\hat{C}$ denote the image of the identity. By \cite[Tag 045J]{stacks-project}, we need to show that the map $\pi$
is flat and the geometric fibers of $\pi$ are Cohen--Macaulay. By a standard reduction, we are reduced to showing that the map $\pi$ is flat at $u$ and the fiber at $u$ is Cohen--Macaulay.
Now we use the fact that locally around $u$, the commuting scheme $\Comm_{\hat{G}}$ over $\overline{\FF}_\ell$ can be approximated by the $q$-commuting scheme
\[
    \Comm^q_{\hat{G}}=\{x,y\in\hat{G} \,|\,xyx^{-1}=y^q\}
\]
as $v_\ell(q-1)$ goes to infinity. More precisely, we have a similar map
\[\pi\colon\Comm^q_{\hat{G}}\to \hat{C},\]
sending $(x,y)$ to the image of $y$. Then there exists a constant $a>0$ depending only on $v_\ell(q-1)$ and $\hat{G}$, such that there is a natural isomorphism
\[
    \Comm_{\hat{G}}\times_{\hat{C}}\hat{C}_a\simeq \Comm^q_{\hat{G}}\times_{\hat{C}}\hat{C}_a,
\]
where $\hat{C}_a$ is the $a$-th nilpotent thickening of $\hat{C}$ at the point $u$. See Corollary \ref{cor:Comm q vs Comm}. Therefore we are reduced to showing that the map
\begin{equation}\label{eq: intro Comm^q C_a}
    \Comm^q_{\hat{G}}\times_{\hat{C}}\hat{C}_a\to \hat{C}_a
\end{equation}
is flat with a Cohen--Macaulay central fiber. Using the unipotent categorical local Langlands established in \cite{ZTame}, this problem can be translated to a problem concerning representations of finite groups of Lie type, as we will now explain.

Let $F=\FF_q(\!(t)\!)$ be a non-Archimedean local field with residue field $\FF_q$. Let $G/\FF_q$ be the split reductive group with dual group $\hat{G}$. We fix a pinning of $G$ and a non-trivial character $\psi\colon \FF_q\to\overline{\FF}_\ell^\times$. In \cite{ZTame}, Zhu established the unipotent categorical Langlands correspondence for $G$. In particular if $\ell$ is sufficiently general, he constructs a fully faithful embedding 
\[
    \LL^{\widehat{\mathrm{unip}}}_{G, \psi, 1}\colon \rep^{\widehat{\mathrm{unip}}}(G(F), \overline{\FF}_\ell) \xhookrightarrow{\quad} \operatorname{IndCoh}(\LS^{\widehat{\mathrm{unip}}}_{\ghat})
\]
of the category $\rep^{\widehat{\unip}}(G(F), \overline{\FF}_\ell)$ of unipotent representations of $G(F)$ into the category of ind-coherent sheaves on the moduli stack $\LS^{\widehat{\mathrm{unip}}}_{\ghat}$ of $\flbar$-coefficient unipotent Langlands parameters valued in the Langlands dual group $\ghat$ of $G$. 

Now let $\rep^{\widehat\unip}_c(G(\FF_q), \flbar)$ denote the subcategory of unipotent representations of $G(\FF_q)$ such that the underlying $\overline{\FF}_\ell$-complex is perfect. In the bulk of this paper we study a finite Langlands functor
\[
    \mathcal{F}_{G,\psi}^{\unip,c}\colon\rep^{\widehat\unip}_c(G(\FF_q), \flbar) \xrightarrow{\quad} \operatorname{Coh}(\LS^{\widehat{\mathrm{unip}}}_{\ghat})
\]
which is a variant of the functor $\mathcal{F}_{G,\psi}^{\widehat{\mathrm{unip}}} \coloneqq \LL_{G,\psi,1}^{\widehat{\mathrm{unip}}} \circ \cind_{G(\cO_F)}^{G(F)}$. Our main result on this functor is the following theorem, which we expect will be of independent interest. This is a generalization of \cite[Theorem 5.10]{ZTame} to modular coefficients.

\begin{thm}[{Theorem \ref{thm: t-exact CM}}]\label{thm: intro t-exact CM}
    Suppose that Assumption \ref{assump: local geom Langlands} holds, and that every elementary $\ell$-subgroup of $G(\FF_q)$ is contained within a torus. Then 
    \begin{enumerate}
        \item The functor 
        \[
            \cF_{G,\psi}^{\unip,c}\colon \Rep^{\widehat\unip}_c(G(\FF_q),\overline{\FF}_\ell)\to \Coh(\LS^{\widehat\unip}_{\hat{G}})
        \]
        is $t$-exact when both sides are endowed with the standard $t$-structures. 
        \item For any object $V\in \Rep_c^{\widehat\unip}(G(\FF_q),\overline{\FF}_\ell)^\heartsuit$, the associated coherent sheaf $\cF_{G,\psi}^{\unip,c}(V)$ is maximal Cohen--Macaulay.
    \end{enumerate}  
\end{thm}

\begin{rmk}
    The assertions of Assumption \ref{assump: local geom Langlands} are expected statements in tame local geometric Langlands. See the discussion below Assumption \ref{assump: local geom Langlands} for a more detailed discussion about the status of these assumptions. In fact we prove, with some additional effort, in Theorem \ref{thm: exact CM large ell} a weaker but unconditional version of Theorem \ref{thm: intro t-exact CM} which suffices for applications to Theorem \ref{thm: intro main theorem}. We also remark that under Assumption \ref{assump: local geom Langlands}, the assumption on $\ell$ in Theorem \ref{thm: intro main theorem} can be relaxed to $\ell\nmid |W|$, where $W$ is the Weyl group of $G$. 
\end{rmk}

Now the stack $\LS^{\widehat\unip}_{\hat{G}}$ is equal to the unipotent connected component in the stacky quotient $\Comm_{\hat{G}}^q/\hat{G}$ of the $q$-commuting scheme. Therefore the fiber of \eqref{eq: intro Comm^q C_a} is Cohen--Macaulay once we find the correct object in $\Rep_c^{\widehat\unip}(G(\FF_q),\overline{\FF}_\ell)^\heartsuit$ corresponding to the structure sheaf on $\LS^{\widehat\unip}_{\ghat}\times_{\hat{C}}\{u\}$, and the flatness of \eqref{eq: intro Comm^q C_a} follows from the flatness of certain $G(\FF_q)$-representations over $\cO(\hat{C}_a)$. See Corollary \ref{cor: Comm^q C_a flat} and Proposition \ref{prop:unipotent LS is CM} for details.

\subsection{Notation and conventions}

We use the language of $\infty$-categories. In particular, derived categories should be understood as derived stable $\infty$-categories. All the functors are derived, except in the following cases:
\begin{enumerate}
    \item If $M,N$ are objects in an abelian category $\mathcal{A}$, we denote by $\Hom_\mathcal{A}(M,N)$ the usual mapping abelian group. However, if $M,N$ are objects in a stable $\infty$-category $\bC$, we will write $R\Hom_\bC(M,N)$ for the mapping complex or spectrum.
    \item If $M,N$ are two modules over a ring $R$, we denote by $M\otimes_RN$ the \emph{classical} tensor product, and denote by $M\otimes^L_RN$ the derived tensor product. Similarly, if $X\to Y\leftarrow Z$ are morphisms of schemes, we denote by $X\times_YZ$ the classical fiber product, and by $X\times^L_Y Z$ the derived fiber product.
\end{enumerate}

\subsection{Tool and computational resource disclosure} We used ChatGPT Pro to proofread this article and for literature search, as well as for some calculations to verify the plausibility of Conjecture \ref{con: main conjecture} in some examples at small primes. The mathematics and writing were done by the authors.

\subsection{Acknowledgments} We would like to thank Sean Cotner and Xinwen Zhu for many helpful and interesting discussions related to the ideas in this work. In addition, the first-named author would like to thank Liang Xiao for the invitation to visit the Beijing International Center for Mathematical Research, where much of the work which went into this article was carried out. The debt this work owes to the ideas of Xinwen Zhu will be obvious to the reader, and is gladly acknowledged.

\section{Preliminaries on representation theory}\label{section: rep theory}

Let $G$ be a split reductive group with connected center over a finite field $\FF_q$. Fix a pinning $(B,T,e)$, where $B$ is a Borel subgroup of $G$, $T\subseteq B$ is a maximal torus, and $e\colon U^-\to \GG_a$ is a non-degenerate additive character on the unipotent radical of the opposite Borel. Let $W$ be the Weyl group of $G$.

Let $\ell$ be a prime not dividing $q$. Let $\Rep(G(\FF_q),\overline{\FF}_\ell)$ denote the derived category of $G(\FF_q)$-representations with coefficients in $\overline{\FF}_\ell$. Let
\[
    \Rep_c(G(\FF_q),\overline{\FF}_\ell)\subseteq \Rep(G(\FF_q),\overline{\FF}_\ell)
\]
be the subcategory of $G(\FF_q)$-representations such that the underlying $\overline\FF_\ell$-complex is perfect. The category $\Rep(G(\FF_q),\overline{\FF}_\ell)$ carries a standard $t$-structure that restricts to a bounded $t$-structure on $\Rep_c(G(\FF_q),\overline{\FF}_\ell)$. Let $\Rep_c(G(\FF_q),\overline{\FF}_\ell)^\heartsuit$ denote the heart of the standard $t$-structure. 

We study the principal block of $\Rep(G(\FF_q),\overline{\FF}_\ell)$ under the assumptions $\ell \mid q-1$ and $\ell\nmid |W|$. Let $T(\FF_q)^{(\ell)}\subseteq T(\FF_q)$ denote the maximal prime-to-$\ell$ subgroup and let $B(\FF_q)^{(\ell)}\subseteq B(\FF_q)$ denote the preimage of $T(\FF_q)^{(\ell)}$ under the quotient map $B(\FF_q)\twoheadrightarrow T(\FF_q)$. Note that $B(\FF_q)^{(\ell)}$ has order prime-to-$\ell$.

\begin{defn}\label{def: R_1}
    Define the $G(\FF_q)$-representations
    \[
        R_1\coloneqq \Ind_{B(\FF_q)}^{G(\FF_q)}\overline{\FF}_\ell,\quad
        \widetilde{R}_1\coloneqq \Ind_{B(\FF_q)^{(\ell)}}^{G(\FF_q)}\overline{\FF}_\ell.
    \]
\end{defn}

\begin{lemma}\label{lemma: tilde R_1 inj proj}
    The $G(\FF_q)$-representation $\widetilde{R}_1$ is projective and injective. The representation $\widetilde{R}_1$ is a finite iterated extension of copies of $R_1$.
\end{lemma}
\begin{proof}
    For $V\in \Rep(G(\FF_q),\overline{\FF}_\ell)^{\heartsuit}$, we have 
    \[\Hom_{G(\FF_q)}(\widetilde{R}_1,V)=V^{B(\FF_q)^{(\ell)}}\quad\text{and} \quad\Hom_{G(\FF_q)}(V,\widetilde{R}_1)=(V_{B(\FF_q)^{(\ell)}})^\vee\]
    by Frobenius reciprocity. Therefore $\Hom(\widetilde{R}_1,-),\Hom(-,\widetilde{R}_1)$ are both $t$-exact as $B(\FF_q)^{(\ell)}$ has order prime to $\ell$. The second assertion is clear as we can write
    \[
        \Ind_{B(\FF_q)^{(\ell)}}^{G(\FF_q)}\overline{\FF}_\ell=\Ind_{B(\FF_q)}^{G(\FF_q)}(\overline{\FF}_\ell[T(\FF_q)/T(\FF_q)^{(\ell)}]),
    \]
    where $\overline{\FF}_\ell[T(\FF_q)/T(\FF_q)^{(\ell)}]$ is a finite iterated extension of copies of $\overline{\FF}_\ell$ as $T(\FF_q)$-representations.
\end{proof}

\begin{lemma}\label{lemma: R_1 ss}
    Assume that $\ell\mid q-1$ and $\ell \nmid |W|$. Then the $G(\FF_q)$-representation $R_1$ is semisimple.
\end{lemma}
\begin{proof}
    First, the endomorphism ring $\End(R_1)=\overline{\FF}_\ell[B(\FF_q)\backslash G(\FF_q)/B(\FF_q)]$ is isomorphic to the group algebra $\overline{\FF}_\ell[W]$ by assumption, and hence is semisimple. Therefore we can write $R_1$ as a direct sum
    \[
        R_1=\bigoplus_{i=1}^n M_i^{\oplus m_i}
    \]
    where $\End(M_i)=\overline{\FF}_\ell$ and $\Hom(M_i,M_j)=0$ for $i\neq j$.
    
    We claim that any irreducible quotient $L$ of $R_1$ is also an irreducible subrepresentation of $R_1$. Indeed, by Frobenius reciprocity we have $L^{B(\FF_q)}\neq 0$. Since $U(\FF_q)$ is a $p$-group, we know that $L^{U(\FF_q)}=L_{U(\FF_q)}$ with the same $T(\FF_q)$-action. Since $T(\FF_q)$ is abelian, a finite dimensional representation of $T(\FF_q)$ with a trivial subrepresentation also has a trivial quotient. Therefore $L_{B(\FF_q)}\neq 0$ and hence there exists an injective map $L\hookrightarrow R_1$.

    Now let $M_i\twoheadrightarrow L$ be an irreducible quotient. By the preceding argument, $L$ is also an irreducible subrepresentation of some  $M_j$. In particular, there is a non-zero map $M_i\hookrightarrow L\twoheadrightarrow M_j$. Hence we necessarily have $i=j$. Then the map $M_i\twoheadrightarrow L\hookrightarrow M_i$ is invertible because $\End(M_i)=\overline{\FF}_\ell$. Therefore $M_i\simeq L$ is simple.
\end{proof}

\begin{defn}
    Assume that $\ell\mid q-1$ and $\ell\nmid |W|$.
    \begin{enumerate}
        \item Define the subcategory 
    \[
        \Rep^0(G(\FF_q),\overline{\FF}_\ell)\subseteq \Rep(G(\FF_q),\overline{\FF}_\ell)
    \]
    to be the category of representations $V$ such that each cohomology group $H^i(V)$ is generated by its $B(\FF_q)^{(\ell)}$-invariants. 
        \item Define the subcategory
    \[
        \Rep^{>0}(G(\FF_q),\overline{\FF}_\ell)=\{V\in \Rep(G(\FF_q),\overline{\FF}_\ell)| V^{B(\FF_q)^{(\ell)}}=0\}.
    \]
    \end{enumerate}
\end{defn}

\begin{prop}\label{prop: principle block generator}
    Assume that $\ell\mid q-1$ and $\ell\nmid |W|$. Then $\Rep^0(G(\FF_q),\overline{\FF}_\ell)$ is a presentable stable subcategory of $\Rep(G(\FF_q),\overline{\FF}_\ell)$, and there is a direct sum decomposition
    \[
        \Rep(G(\FF_q),\Lambda)=\Rep^0(G(\FF_q),\Lambda)\oplus \Rep^{>0}(G(\FF_q),\Lambda)
    \]
    of presentable stable $\overline{\FF}_\ell$-linear categories.
\end{prop}
\begin{proof}
    We first apply Lemma \ref{lemma: decomp of abelian cat} to $\mathcal{A}=\Rep_c(G(\FF_q),\Lambda)^\heartsuit$ and $P=\widetilde{R}_1$. Conditions (1) and (2) of Lemma \ref{lemma: decomp of abelian cat} follow from Lemma \ref{lemma: tilde R_1 inj proj} and Lemma \ref{lemma: R_1 ss}. We see that there is a direct sum decomposition 
    \begin{equation}\label{eq: principal decomp abelian cat}
        \Rep_c(G(\FF_q),\Lambda)^\heartsuit= \Rep^0_c(G(\FF_q),\Lambda)^\heartsuit\oplus \Rep^{>0}_c(G(\FF_q),\Lambda)^\heartsuit
    \end{equation}        
    of abelian categories. Note that $\Rep(G(\FF_q),\overline{\FF}_\ell)^\heartsuit$ admits a projective generator $Q=\overline{\FF}_\ell[G(\FF_q)]\in \Rep_c(G(\FF_q),\Lambda)^\heartsuit$. By \eqref{eq: principal decomp abelian cat}, we have a direct sum decomposition
    \[
        Q=Q^0\oplus Q^{>0}
    \]
    with $Q^0\in \Rep^0_c(G(\FF_q),\Lambda)^\heartsuit$ and $Q^{>0}\in \Rep^{>0}_c(G(\FF_q),\Lambda)^\heartsuit$ such that $\Hom(Q^0,Q^{>0})=\Hom(Q^{>0},Q^0)=0$. Then 
    \[
    \Rep(G(\FF_q),\Lambda)=\langle Q^0\rangle \oplus \langle Q^{>0}\rangle
    \]
    where $\langle Q^0\rangle$ (resp. $\langle Q^{>0}\rangle$) is the presentable stable subcategory generated by $Q^0$ (resp. $Q^{>0}$). It is clear that $\langle Q^0\rangle=\langle \widetilde{R}_1\rangle=\Rep^0(G(\FF_q),\Lambda)$. Then $\langle Q^{>0}\rangle$ is identified with $\Rep^{>0}(G(\FF_q),\Lambda)$ consisting of objects $V$ with $\Hom_{G(\FF_q)}(V,\widetilde{R}_1)=0$.
\end{proof}

\begin{lemma}\label{lemma: decomp of abelian cat}
    Let $\mathcal{A}$ be a finite-length (i.e. Artinian and Noetherian) abelian category, and let $P\in\mathcal{A}$ satisfy the following conditions:
    \begin{enumerate}
        \item The object $P$ is both injective and projective.
        \item For any simple Jordan--H\"older factor $L$ of $P$, there exists a non-zero surjective map $P\twoheadrightarrow L$.
    \end{enumerate}
    Let $\mathcal{A}_P$ be the Serre subcategory of $\mathcal{A}$ generated by simple Jordan--H\"older factors of $P$. Let $\mathcal{A}^P$ be the subcategory of $\mathcal{A}$ consisting of objects $X\in\mathcal{A}$ such that $\Hom(X,P)=0$. Then there is a direct sum decomposition
    \[
        \mathcal{A}=\mathcal{A}_P\oplus \mathcal{A}^P
    \]
    of abelian categories.
\end{lemma}
\begin{proof}
     We can write
    \[
        P=\bigoplus_{i=1}^n P_i^{\oplus m_i}
    \]
    where $P_i$ are pairwise non-isomorphic indecomposable objects that are both injective and projective. Each $P_i$ has a unique simple quotient $L_i$ and a unique simple subobject $L_i'$. Moreover, $L_1,\dots,L_n$ are pairwise non-isomorphic and $L_1',\dots,L_n'$  are pairwise non-isomorphic. By (2), it follows that
    \[
        S=\{L_1,\dots,L_n\}=\{L_1',\dots,L_n'\}
    \]
    is the set of pairwise non-isomorphic irreducible Jordan--H\"older factors of $P$, and $\mathcal{A}_P$ is the Serre subcategory generated by $S$. In particular, for $i=1,\dots, n$, there exists an injective map $L_i\hookrightarrow P$. 

    By injectivity of $P$, every nonzero object $X\in \mathcal{A}_P$ admits a nonzero map $X\to P$. Therefore $\mathcal{A}^P\cap \mathcal{A}_P=0$. In particular, we have
    \[
        \Hom_{\mathcal{A}}(X,Y)=0=\Hom_{\mathcal{A}}(Y,X)
    \]
    for any $X\in\mathcal{A}_P,Y\in\mathcal{A}^P$.
    
    Let $M\in \mathcal{A}$. Consider
    \[
        Y=\bigcap_{f\colon X\to P}\ker(f).
    \]
    Since $X$ has finite length, this intersection equals a finite intersection. Let $Z=X/Y$. We have a short exact sequence
    \[
        0\to Y\to X\to Z\to 0.
    \]
    If there exists a non-zero map $Y\to P$, then by injectivity of $P$ we can extend it to a map $X\to P$, contradicting the definition of $Y$. Therefore $\Hom(Y,P)=0$ and hence $Y\in \mathcal{A}^P$. On the other hand, $Z$ embeds into a finite direct sum of copies of $P$. Therefore the simple Jordan--H\"older factors of $Z$ belong to $S$, and hence $Z\in \mathcal{A}_P$. Now it suffices to show that the extension splits. By projectivity of $P$, we can find a short exact sequence
    \[
        0\to W\to P^{\oplus m}\to Z\to 0
    \]
    with $W\in\mathcal{A}_P$. It induces an exact sequence
    \[
        \Hom_{\mathcal{A}}(W,Y) \to \mathrm{Ext}^1_{\mathcal{A}}(Z,Y)\to \mathrm{Ext}^1_{\mathcal{A}}(P^{\oplus m},Y).
    \]
    However, we know that $\Hom_{\mathcal{A}}(W,Y)=0$ since $W\in\mathcal{A}_P$, and $\mathrm{Ext}^1_{\mathcal{A}}(P^{\oplus m},Y)=0$ by projectivity of $P$. Therefore every extension of $Z$ by $Y$ splits.
\end{proof}

Fix a non-trivial character $\psi\colon \FF_q\to \overline{\FF}_\ell^\times$. 
\begin{defn}
    Define the Gelfand--Graev representation to be
    \[
        \Gamma_\psi\coloneqq \Ind_{U^-(\FF_q)}^{G(\FF_q)}\psi\circ e.
    \]
    Here $\psi\circ e\colon U^-(\FF_q)\to\overline{\FF}_\ell^\times$ is the resulting non-degenerate character.
\end{defn}

Note that $\Gamma_\psi$ is a compact projective object in $\Rep(G(\FF_q),\overline{\FF}_\ell)^\heartsuit$.

Let $\hat{G}$ denote the dual group of $G$ over $\overline{\FF}_\ell$. Let $\hat{T}\subseteq \hat{G}$ be the dual maximal torus. Let 
\[
\hat{C}\coloneqq \hat{G}\git \hat{G}\simeq \hat{T}\git W
\]
denote the Chevalley quotient. Let $[q]$ denote the $q$-power maps on $\hat{G}$ and $\hat{T}$. We also denote by $[q]\colon\hat{C}\to\hat{C}$ the map induced by the $q$-power map on $\hat{G}$. Let $\hat{C}^{[q]}$ denote the subscheme of  $[q]$-fixed points. The scheme $\hat{C}^{[q]}$ is finite over $\overline{\FF}_\ell$. Let $\hat{C}^{[q]}_u$ denote the connected component containing the identity. Similarly, let $\hat{T}^{[q]}$ denote the subscheme of $[q]$-fixed points. Let $\hat{T}^{[q]}_u$ denote the connected component containing the identity. There is a natural isomorphism
\[
    \cO(\hat{T}^{[q]}_u)\simeq \overline{\FF}_\ell[T(\FF_q)/T(\FF_q)^{(\ell)}]
\]
of $\overline{\FF}_\ell$-algebras.

By \cite[\S 5.2]{Eteve-Jordan}, the category $\Rep(G(\FF_q),\overline{\FF}_\ell)$ is naturally an $\cO(\hat{C}^{[q]})$-linear category. The construction of the $\cO(\hat{C}^{[q]})$-action will be reviewed in Proposition \ref{prop: F_G linear}. In particular, we obtain a direct sum decomposition
\[  
    \Rep(G(\FF_q),\overline{\FF}_\ell)=\bigoplus_{\mathfrak{s}\in\pi_0(\hat{C}^{[q]})}\Rep^{\mathfrak{s}}(G(\FF_q),\overline{\FF}_\ell).
\]
Consider the category of unipotent representations
\[
    \Rep^{\widehat\unip}(G(\FF_q),\overline{\FF}_\ell)\coloneqq \Rep^{\hat{C}^{[q]}_u}(G(\FF_q),\overline{\FF}_\ell).
\]
By \cite[Lemma 5.3.1.(i)]{Eteve-Jordan}, the representation $\widetilde{R}_1$  lies in $\Rep^{\widehat\unip}(G(\FF_q),\overline{\FF}_\ell)$. Therefore
\begin{equation}\label{eq: 0 in unip}
    \Rep^0(G(\FF_q),\overline{\FF}_\ell)\subseteq \Rep^{\widehat\unip}(G(\FF_q),\overline{\FF}_\ell).
\end{equation}
Denote
\[
    \Rep_c^{\widehat\unip}(G(\FF_q),\overline{\FF}_\ell)= \Rep^{\widehat\unip}(G(\FF_q),\overline{\FF}_\ell)\cap \Rep_c(G(\FF_q),\overline{\FF}_\ell).
\]

By \cite[Theorem 6.1.3]{Eteve-Jordan}, the natural map
\[
    \cO(\hat{C}^{[q]})\to \End_{G(\FF_q)}(\Gamma_\psi)
\]
is an isomorphism. We see that there is a direct sum decomposition $\Gamma_\psi =\bigoplus_{\mathfrak{s}\in \pi_0(\hat{C}^{[q]})}\Gamma_{\psi}^\mathfrak{s}$ with each $\Gamma_{\psi}^\mathfrak{s}$ an indecomposable projective in $\Rep_c(G(\FF_q),\overline{\FF}_\ell)^\heartsuit$. Let 
\begin{equation}\label{eq: unip GG}
    \Gamma_\psi^{\widehat\unip}\coloneqq \Gamma_\psi^{\hat{C}^{[q]}_u}
\end{equation}
denote the direct summand in $\Rep^{\widehat\unip}(G(\FF_q),\overline{\FF}_\ell)$. It is clear that
\begin{equation}\label{eq: End Gamma_psi^unip}
    \End_{G(\FF_q)}(\Gamma_\psi^{\widehat\unip})\simeq \cO(\hat{C}^{[q]}_u).
\end{equation}

\begin{prop}\label{prop:GGinRep0}
    Assume that $\ell\mid q-1$ and $\ell\nmid |W|$. The representation $\Gamma_{\psi}^{\widehat\unip}$ lies in the category $\Rep^0_c(G(\fq), \overline{\FF}_\ell)^{\heartsuit}$. 
\end{prop}
\begin{proof}
    By \cite[Theorem 6.1.2]{Eteve-Jordan}, we have
    \[
        \Hom_{G(\FF_q)}(\Ind_{U(\FF_q)}^{G(\FF_q)}\overline{\FF}_\ell,\Gamma_\psi)\simeq \cO(\hat{T}^{[q]})
    \]
    as $\cO(\hat{C}^{[q]})$-modules. It is clear that $\widetilde{R}_1$ is equal to the direct summand of $\Ind_{U(\FF_q)}^{G(\FF_q)}\overline{\FF}_\ell$ in $\Rep^{\widehat\unip}(G(\FF_q),\overline{\FF}_\ell)$. We see that
    \begin{equation}\label{eq: Hom R_e, Gamma}
        \Hom_{G(\FF_q)}(\widetilde{R}_1,\Gamma_{\psi}^{\widehat\unip})\simeq \cO(\hat{T}^{[q]}_u),
    \end{equation}
    which is in particular nonzero. As $\Gamma_{\psi}^{\widehat\unip}$ is indecomposable, we  have $\Gamma_{\psi}^{\widehat\unip}\in \Rep^0(G(\FF_q),\Lambda)$ by Proposition \ref{prop: principle block generator}.   
\end{proof}

\section{Finite Langlands functor}
\subsection{Tame categorical local Langlands}
Let $F=\FF_q(\!(t)\!)$ be a non-Archimedean local field. Let $O_F$ be the ring of integers in $F$. Let $G$ be a split reductive group over $\FF_q$. By abuse of notation, let $G$ denote the split reductive group over $O_F$ obtained from $G/\FF_q$ by base change. Let $W_F$ denote the Weil group of $F$. Let $\LS^\tame_{\hat{G}}$ denote the stack of tame local Langlands parameters of $\hat{G}$ over $\overline{\FF}_\ell$ defined in \cite[\S 2.2.1]{ZTame}. After choosing a lift of the arithmetic Frobenius $\sigma\in W_F$ and a topological generator $\tau$ of the tame inertia, we obtain an isomorphism
\begin{equation}\label{eq: LS= C^q/G}
    \LS^\tame_{\hat{G}}\xrightarrow{\sim} \Comm^q_{\hat{G}}/\hat{G},
\end{equation}
where $\Comm^q_{\hat{G}}$ is the $q$-commuting scheme
\[\Comm^q_{\hat{G}}\coloneqq \{ x,y\in \hat{G}\,|\, xyx^{-1}=y^q\}\]
over $\overline{\FF}_\ell$.
The map \eqref{eq: LS= C^q/G} is defined by sending a Langlands parameter $\varphi\colon W_F\to\hat{G}$ to $(\varphi(\sigma),\varphi(\tau))$. Note that there is a natural map 
\begin{equation}\label{eq: LS -> C^[q]}
    \LS^\tame_{\hat{G}}\to \hat{C}^{[q]}
\end{equation}
sending $(x,y)\in\Comm^q_{\hat{G}}$ to the image of $y$ in the Chevalley quotient. Let $\LS^{\widehat\unip}_{\hat{G}}\subseteq \LS^\tame_{\hat{G}}$ denote the open and closed substack defined by the preimage of $\hat{C}^{[q]}_u\subseteq \hat{C}^{[q]}$. Let $\Coh(\LS^\tame_{\hat{G}})$ denote the derived category of coherent sheaves on $\LS^\tame_{\hat{G}}$. Let 
\[
\DD_\GS=\underline{R\Hom}(-,\omega_{\LS^\tame_{\hat{G}}})\colon \Coh(\LS^\tame_{\hat{G}})^\op\xrightarrow{\sim} \Coh(\LS^\tame_{\hat{G}}).
\]
denote the Grothendieck--Serre duality. Let $c\colon \LS^\tame_{\hat{G}}\xrightarrow{\sim}\LS^\tame_{\hat{G}}$ denote the involution defined by the Chevalley involution on $\hat{G}$. Let
\[
    \DD_\GS'\coloneqq c^*\circ \DD_\GS\colon \Coh(\LS^\tame_{\hat{G}})^\op\xrightarrow{\sim} \Coh(\LS^\tame_{\hat{G}})
\]
denote the modified Grothendieck--Serre duality.

Let $\Rep(G(F),\overline{\FF}_\ell)$ denote the derived category of $G(F)$-representations. Let $\Rep_\fg(G(F),\overline{\FF}_\ell)\subseteq \Rep(G(F),\overline{\FF}_\ell)$ be the subcategory of finitely generated representations in \cite[Proposition 3.57]{ZTame}. Recall that $\Rep_\fg(G(F),\overline{\FF}_\ell)$ contains the subcategory $\Rep(G(F),\overline{\FF}_\ell)^\omega$ of compact objects. By \cite[Corollary 3.62]{ZTame}, there is a canonical cohomological duality
\begin{equation}\label{eq: D_coh^fg}
    \DD_\coh^\fg\colon \Rep_\fg(G(F),\overline{\FF}_\ell)^\op\xrightarrow{\sim} \Rep_\fg(G(F),\overline{\FF}_\ell)
\end{equation}
such that for an open compact subgroup $K\subseteq G(F)$ and a $K$-representation $V$ whose underlying $\overline{\FF}_\ell$-complex is perfect, there is a canonical isomorphism
\begin{equation}\label{eq: D cInd}
    \DD_\coh^\fg(\cInd_K^{G(F)}V)\simeq \cInd_K^{G(F)} V^\vee.
\end{equation}
Here $V^\vee=R\Hom_{\overline{\FF}_\ell}(V,\overline{\FF}_\ell)$ with the natural $K$-action. 

Let 
\begin{equation}\label{eq: Rep u t}
    \Rep^{\widehat\unip}(G(F),\overline{\FF}_\ell)\subseteq\Rep^{\tame}(G(F),\overline{\FF}_\ell)\subseteq \Rep(G(F),\overline{\FF}_\ell)
\end{equation}
denote, respectively, the subcategories of unipotent representations and depth zero representations defined in \cite[Definition 4.105 and Definition 4.116]{ZTame}. The categories in \eqref{eq: Rep u t} carry standard $t$-structures.

Let $\Rep_\fg^\unip(G(F),\overline{\FF}_\ell)\subseteq \Rep^{\widehat\unip}(G(F),\overline{\FF}_\ell)$ denote the subcategory of finitely generated unipotent representations, as defined in \cite[Definition 4.116]{ZTame}. By \cite[Remark 4.117]{ZTame} we have
\[
    \Rep^{\widehat\unip}(G(F),\overline{\FF}_\ell)^\omega\subseteq \Rep_\fg^\unip(G(F),\overline{\FF}_\ell).
\]
By \cite[Proposition 4.122]{ZTame}, the subcategory $\Rep^{\unip}_\fg(G(F),\overline{\FF}_\ell)$ is stable under the duality \eqref{eq: D_coh^fg}. Hence there is a natural duality
\[
    \DD_\coh^\fg\colon \Rep_\fg^\unip(G(F),\overline{\FF}_\ell)^\op\xrightarrow{\sim} \Rep_\fg^\unip(G(F),\overline{\FF}_\ell).
\]
Let $\Rep_c^{\unip}(G(\FF_q),\overline{\FF}_\ell)$ denote the category defined in \cite[Definition 4.86]{ZTame}, which embeds fully faithfully as a subcategory
\begin{equation}\label{eq: u' in u}
    \Rep_c^{\unip}(G(\FF_q),\overline{\FF}_\ell)\subseteq \Rep^{\widehat\unip}_c(G(\FF_q),\overline{\FF}_\ell).
\end{equation}
The category $\Rep_c^{\unip}(G(\FF_q),\overline{\FF}_\ell)$ is stable under the duality $V\mapsto V^\vee$.
By definition, the composition functor
\begin{equation}\label{eq: cInd infl}
    \Rep_c(G(\FF_q),\overline{\FF}_\ell)\to \Rep_c(G(O_F),\overline{\FF}_\ell)\xrightarrow{\cInd_{G(O_F)}^{G(F)}} \Rep_\fg(G(F),\overline{\FF}_\ell)
\end{equation}
sends the subcategory $\Rep^{\unip}_c(G(\FF_q),\overline{\FF}_\ell)$ to the subcategory $\Rep^{\unip}_\fg(G(F),\overline{\FF}_\ell)$. In a mild abuse of notation, we denote by 
\[
    \cInd_{G(O_F)}^{G(F)}\colon \Rep^{\unip}_c(G(\FF_q),\overline{\FF}_\ell)\to \Rep^{\unip}_\fg(G(F),\overline{\FF}_\ell).
\]
the resulting functor. We note that \eqref{eq: cInd infl} always sends the subcategory $\Rep(G(\FF_q),\overline{\FF}_\ell)^\omega$ to $\Rep^{\tame}(G(F),\overline{\FF}_\ell)^\omega$ and the subcategory $\Rep^{\widehat\unip}(G(\FF_q),\overline{\FF}_\ell)^\omega$ to $\Rep^{\widehat\unip}(G(F),\overline{\FF}_\ell)^\omega$.

We now state the geometric input needed for the tame categorical local Langlands correspondence. Recall that we have fixed a pinning on $G$ and a non-trivial character $\psi\colon\FF_q\to\overline{\FF}_\ell^\times$.

\begin{assump}\label{assump: local geom Langlands}
    Assume that $\ell\nmid \pi_1(\hat{G})_{\mathrm{tor}}$. We assume the following statements concerning the tame geometric local Langlands correspondence.
    \begin{enumerate}[(a)]
        \item \cite[Theorem 5.1.(1)(2)(4)(5)]{ZTame} holds with $\Lambda=\overline{\FF}_\ell$. Namely, there exist equivalences of monoidal categories
        \begin{equation}\label{eq: B_G tame}
        \BB_{G,\psi}^\tame\colon \Shv_\mon(\mathrm{Iw}^u\backslash LG/\mathrm{Iw}^u,\overline{\FF}_\ell)\simeq \Ind\Coh(\LS_{\hat{B},\breve{F}}^\tame\times_{\LS_{\hat{G},\breve{F}}^\tame}\LS_{\hat{B},\breve{F}}^\tame)
        \end{equation}
        \begin{equation}\label{eq: B_G hatunip}
        \BB_{G,\psi}^{\widehat{\unip}}\colon \Shv_{u-\mon}(\mathrm{Iw}^u\backslash LG/\mathrm{Iw}^u,\overline{\FF}_\ell)\simeq \Ind\Coh(\LS_{\hat{B},\breve{F}}^{\widehat\unip}\times_{\LS_{\hat{G},\breve{F}}^{\widehat\unip}}\LS_{\hat{B},\breve{F}}^{\widehat\unip})
        \end{equation}
        \begin{equation}\label{eq: B_G unip}
        \BB_{G,\psi}^\unip\colon \Ind\Shv_\fg(\mathrm{Iw}\backslash LG/\mathrm{Iw},\overline{\FF}_\ell)\simeq \Ind\Coh(\LS_{\hat{B},\breve{F}}^\unip\times^L_{\LS_{\hat{G},\breve{F}}^\tame}\LS_{\hat{B},\breve{F}}^\unip)
        \end{equation}
        with the compatibilities specified in the cited theorem. For the notation, see \cite[\S 2.3 and \S 4.2]{ZTame}.
        \item The equivalence \eqref{eq: B_G tame} (resp. \eqref{eq: B_G hatunip}) is compatible with the actions of $\Shv_\mon(T,\Lambda)\simeq\Ind\Coh(R_{\hat{T}})$ (resp. $\Shv_{u-\mon}(T,\Lambda)\simeq\Ind\Coh(R_{\hat{T},\hat{u}})$ action) on both sides. See \cite[Proposition 4.32]{ZTame} for the notation.
        \item Let $\mathcal{Z}^\tame(V)$ be the central sheaf in \cite[Theorem 5.1.(2)]{ZTame} where $V$ is any tilting $\hat{G}$-representation. Then the object
        \[
            \mathcal{Z}^\tame(V)\star^u \widetilde{\Delta}^{\mon,\psi}_{w_0}\in\Shv_\mon(\mathrm{Iw}^u\backslash LG/(\mathrm{Iw}^u,\psi),\overline{\FF}_\ell)
        \]
        in \cite[Theorem 4.136]{ZTame} is cofree tilting. We assume the analogous statement in the unipotent monodromic setting.
    \end{enumerate}
\end{assump}
\begin{rmk}\label{rmk: on assump local geom Langlands}
    We summarize the status of Assumption \ref{assump: local geom Langlands}. If $\ell$ is larger than the Coxeter number of any simple factors of $G$, and $\ell\neq 19$ (resp. $\ell\neq 31$) when $G$ has a simple factor of type $E_7$ (resp. $E_8$), then the equivalences \eqref{eq: B_G hatunip} and \eqref{eq: B_G unip} are constructed in \cite{BR-modular-two-affine-Hecke}. Moreover, these equivalences satisfy \cite[Theorem  5.1.(2)]{ZTame} by \cite[Lemma 11.3]{BR-modular-two-affine-Hecke}. The unipotent monodromic case of (b) is also proved in \cite{BR-modular-two-affine-Hecke}, and the unipotent monodromic version of (c)  is proved in \cite[Proposition 7.9]{BR-modular-affine-Hecke}. We note that for our applications to commuting schemes, we only need the unipotent monodromic version.

    In general, Assumption \ref{assump: local geom Langlands} will be proved in a work in progress of Xinyu Li and Jiahao Niu.
\end{rmk}

We recall the following constructions in \cite{ZTame}.
\begin{thm}\label{thm: cllc}
    Suppose that Assumption \ref{assump: local geom Langlands} holds.
    \begin{enumerate}
        \item There is a fully faithful functor
    \[
        \LL_{G,\psi,1}^{\tame}\colon \Rep^{\tame}(G(F),\overline{\FF}_\ell)\hookrightarrow  \Ind\Coh(\LS^\tame_{\hat{G}}). 
    \]
    that restricts to a fully faithful embedding
    \begin{equation}\label{eq: L hat unip}
        \LL_{G,\psi,1}^{\widehat\unip}\colon \Rep^{\widehat\unip}(G(F),\overline{\FF}_\ell)\hookrightarrow \Ind\Coh(\LS^{\widehat\unip}_{\hat{G}}). 
    \end{equation} 
    Moreover, the functor $\LL_{G,\psi,1}^{\tame}$ preserves compact objects.
        \item We have
        \[\LL^{\tame}_{G,\psi,1}(\cInd_{G(O_F)}^{G(F)}\Gamma_\psi)=\cO_{\LS^\tame_{\hat{G}}}.\]
        \item The functor $\LL_{G,\psi,1}^{\widehat\unip}|_{\Rep^{\widehat\unip}(G(F),\overline{\FF}_\ell)^\omega}$ can be extended to a fully faithful embedding
    \begin{equation}\label{eq: L ufg}
        \LL_{G,\psi,1}^{\unip,\fg}\colon \Rep^{\unip}_\fg(G(F),\overline{\FF}_\ell)\to \Coh(\LS^{\widehat\unip}_{\hat{G}}). 
    \end{equation}
        such that
        \begin{equation}\label{eq: L^unip }
             \LL_{G,\psi^{-1},1}^{\unip,\fg}\circ \DD_{\coh}^\fg\simeq \DD'_\GS\circ \LL_{G,\psi,1}^{\unip,\fg}.
         \end{equation}
    \end{enumerate}
\end{thm}
\begin{proof}
    By the proof of \cite[Theorem 5.4.(1)]{ZTame}, we have a fully faithful embedding
    \[
        \LL^{\tame}_{G,\psi}\colon \Shv^\tame(\Isoc_G,\overline{\FF}_\ell)\hookrightarrow\Ind\Coh(\LS^\tame_{\hat{G}})
    \]
    that restricts to a fully faithful embedding
    \[
        \LL^{\widehat\unip}_{G,\psi}\colon \Shv^{\widehat\unip}(\Isoc_G,\overline{\FF}_\ell)\hookrightarrow\Ind\Coh(\LS^{\widehat\unip}_{\hat{G}})
    \] 
    where the categories $\Shv^{\tame}(\Isoc_G,\Lambda),\Shv^{\widehat\unip}(\Isoc_G,\Lambda)$ are defined in \cite[Definition 4.106 and Definition 4.119]{ZTame}. In particular, there are fully faithful embeddings
    \[
        (i_1)_*\colon \Rep^{?}(G(F),\overline{\FF}_\ell)^\omega\hookrightarrow\Shv^{?}(\Isoc_G,\overline{\FF}_\ell)^\omega
    \]
    for $?\in\{\emptyset,\widehat\unip,\tame\}$ by \cite[Proposition 3.69]{ZTame}. We define the functors in (1) to be
    \[
        \LL^{\tame}_{G,\psi,1}\coloneqq\LL^{\tame}_{G,\psi}\circ (i_1)_*,\quad \LL^{\widehat\unip}_{G,\psi,1}\coloneqq \LL^{\widehat\unip}_{G,\psi}\circ (i_1)_*
    \]
    This proves (1). Then (2) follows from the proof of \cite[Theorem 5.4.(2)]{ZTame}.

    By the proof of \cite[Theorem 5.5]{ZTame}, the functor $\LL_{G,\psi}^{\widehat\unip}|_{\Shv^{\widehat\unip}(\Isoc_G,\overline{\FF}_\ell)^\omega}$ extends to a fully faithful embedding
    \[
        \LL^{\unip,\fg}_{G,\psi}\colon\Shv^\unip_\fg(\Isoc_G,\Lambda)\hookrightarrow\Coh(\LS^{\widehat\unip}_{\hat{G}}).
    \]
    for $\Shv^\unip_\fg(\Isoc_G,\Lambda)$ defined in \cite[Definition 4.119]{ZTame}.
    By \cite[Proposition 3.94]{ZTame}, there is a fully faithful embedding
    \[
        (i_1)_*^{\fg}\colon\Rep^{\unip}_\fg(G(F),\overline{\FF}_\ell)\hookrightarrow \Shv^\unip_\fg(\Isoc_G,\overline{\FF}_\ell).
    \]
    The functor in (3) is therefore defined by
    \[
        \LL_{G,\psi,1}^{\unip,\fg}\coloneqq \LL^{\unip,\fg}_{G,\psi}\circ (i_1)_*^\fg.
    \]
    The compatibility with dualities follows from \cite[Theorem 5.4.(4)]{ZTame} and \cite[Corollary 3.101]{ZTame}, using that $(i_1)^\fg_!\simeq (i_1)^\fg_*$ as $i_1$ is a closed embedding.
\end{proof}

\subsection{Langlands functor for finite groups}
\begin{defn}
    Suppose that Assumption \ref{assump: local geom Langlands} holds. Define the functors
    \begin{equation}\label{eq: F G}
        \cF_{G,\psi}=\LL^\tame_{G,\psi,1}\circ \cInd_{G(O_F)}^{G(F)} \colon \Rep(G(\FF_q),\overline{\FF}_\ell)\to \Ind\Coh(\LS^\tame_{\hat{G}}),
    \end{equation}
    \begin{equation}\label{eq: F hatunip}
        \cF_{G,\psi}^{\widehat\unip}=\LL^{\widehat\unip}_{G,\psi,1}\circ \cInd_{G(O_F)}^{G(F)} \colon \Rep^{\widehat\unip}(G(\FF_q),\overline{\FF}_\ell)\to \Ind\Coh(\LS^{\widehat\unip}_{\hat{G}}),
    \end{equation}
    and
    \begin{equation}\label{eq: F unip c}
         \cF_{G,\psi}^{\unip,c}=\LL^{\unip,\fg}_{G,\psi,1}\circ \cInd_{G(O_F)}^{G(F)} \colon \Rep_c^{\unip}(G(\FF_q),\overline{\FF}_\ell)\to \Coh(\LS^{\widehat\unip}_{\hat{G}}).
    \end{equation}
\end{defn}

By \eqref{eq: D cInd} and Theorem \ref{thm: cllc}.(3), we have
\begin{equation}\label{eq: D_GS F^c}
    \DD_{\GS}'(\cF_{G,\psi}^{\unip,c}(V))\simeq \cF_{G,\psi^{-1}}^{\unip,c}(V^\vee)
\end{equation}
for $V\in\Rep_c^\unip(G(\FF_q),\overline{\FF}_\ell)$.

We warn the reader that $\cF_{G,\psi}^{\unip,c}$ is \emph{not} the restriction of $\cF_{G,\psi}^{\widehat\unip}$ to the subcategory $\Rep_c^{\unip}(G(\FF_q),\overline{\FF}_\ell)\subseteq \Rep^{\widehat\unip}(G(\FF_q),\overline{\FF}_\ell)$. However, $\cF_{G,\psi}^{\unip,c}$ and $\cF_{G,\psi}^{\widehat\unip}$ agree when restricted to the category $\Rep^{\widehat\unip}(G(\FF_q),\overline{\FF}_\ell)^\omega$ of compact objects.

The category $\Rep(G(\FF_q),\overline{\FF}_\ell)$ is $\cO(\hat{C}^{[q]})$-linear by \cite[\S 5.2]{Eteve-Jordan}. On the spectral side, the morphism \eqref{eq: LS -> C^[q]} endows the category  $\Ind\Coh(\LS^\tame_{\hat{G}})$ with an $\cO(\hat{C}^{[q]})$-linear structure. 

\begin{prop}\label{prop: F_G linear}
    Suppose that Assumption \ref{assump: local geom Langlands} holds. The functor $\cF_{G,\psi}$ is $\cO(\hat{C}^{[q]})$-linear. Similarly, the functor $\cF_{G,\psi}^{\widehat\unip}$ is $\cO(\hat{C}^{[q]}_u)$-linear.
\end{prop}
\begin{proof}
    The $\cO(\hat{C}^{[q]})$-action is defined by taking the categorical trace of the $\cO(\hat{C})$-action on the monodromic finite Hecke category $\Shv_\mon(U\backslash G/U,\overline{\FF}_\ell)$ (see \cite[\S 4]{ZTame} for the definitions). More precisely, after choosing the generator $\tau$ of the tame inertia subgroup $I_F$, the $T\times T$-monodromic action defines an $\cO(\hat{T}\times\hat{T})$-action on the category $\Shv_\mon(U\backslash G/U,\overline{\FF}_\ell)$, which factors through $\cO(\hat{T}\times_{\hat{C}}\hat{T})$ by \cite[\S 3.5.3]{Eteve-Jordan}. Then $\Shv_\mon(U\backslash G/U,\overline{\FF}_\ell)$ is naturally a $\cO(\hat{C})$-linear monoidal category. Then the Frobenius categorical trace (\cite[Theorem 1.3.1]{Eteve2024FreeMonodromic} and \cite[Theorem 4.97]{ZTame})
    \[\Rep(G(\FF_q),\overline{\FF}_\ell)\simeq \mathrm{Tr}(\Shv_\mon(U\backslash G/U,\overline{\FF}_\ell),\sigma_*)\]
    carries an action of $\mathrm{Tr}(\QCoh(\hat{C}),[q]^*)=\QCoh(\hat{C}^{[q]})$ (\cite[Lemma 5.2.1]{Eteve-Jordan}), and hence is $\cO(\hat{C}^{[q]})$-linear.

    The same argument applies to the monodromic affine Hecke category $\Shv_\mon(\mathrm{Iw}^u\backslash LG/\mathrm{Iw}^u,\overline{\FF}_\ell)$ and implies that the categorical trace (\cite[Theorem 4.125]{ZTame})
    \begin{equation}\label{eq: Isoc=tr}
        \Shv^\tame(\Isoc_G,\overline{\FF}_\ell)\simeq \mathrm{Tr}(\Shv_\mon(\mathrm{Iw}^u\backslash LG/\mathrm{Iw}^u,\overline{\FF}_\ell),\sigma_*)
    \end{equation}
    is $\cO(\hat{C}^{[q]})$-linear. Moreover, the functor
    \begin{equation}\label{eq: i_1 cInd}
        (i_1)_*\circ \cInd_{G(O_F)}^{G(F)}\colon \Rep(G(\FF_q),\overline{\FF}_\ell)\to \Shv^\tame(\Isoc_G,\overline{\FF}_\ell)
    \end{equation}
    is identified with the categorical trace of the $\cO(\hat{C})$-linear monoidal functor
    \begin{equation}\label{eq: finite Hecke in affine Hecke}
        \Shv_\mon(U\backslash G/U,\overline{\FF}_\ell)\simeq\Shv_\mon(\mathrm{Iw}^u\backslash L^+G/\mathrm{Iw}^u,\overline{\FF}_\ell)\hookrightarrow\Shv_\mon(\mathrm{Iw}^u\backslash LG/\mathrm{Iw}^u,\overline{\FF}_\ell)
    \end{equation}        
    by \cite[Lemma 4.67]{ZTame}. Therefore \eqref{eq: i_1 cInd} is $\cO(\hat{C}^{[q]})$-linear.

    Now the fully faithful embedding 
    \[
        \LL_{G,\psi}^\tame\colon\Shv^\tame(\Isoc_G,\overline{\FF}_\ell)\hookrightarrow\Ind\Coh(\LS^\tame_{\hat{G}})
    \]
    is defined by taking the Frobenius categorical trace of \eqref{eq: B_G tame}. After choosing $\tau$, we can identify $\LS^\tame_{\hat{G},\breve{F}}$ with the completion of $\hat{G}/\hat{G}$ along fibers of $\hat{G}/\hat{G}\to\hat{C}$ at points in $\hat{C}(\overline{\FF}_\ell)$ that can be lifted to a point of $\hat{T}(\overline{\FF}_\ell)$ of order prime to $\ell q$, and we can identify $\LS^\tame_{\hat{B},\breve{F}}$ with the completion of $\hat{B}/\hat{B}$ along fibers of $\hat{G}/\hat{G}\to\hat{C}$ at points in $\hat{T}(\overline{\FF}_\ell)$ of order prime to $\ell q$. In particular, there is a natural map
    \[
        \LS_{\hat{B},\breve{F}}^\tame\times_{\LS_{\hat{G},\breve{F}}^\tame}\LS_{\hat{B},\breve{F}}^\tame\to \hat{T}\times_{\hat{C}}\hat{T}
    \]
    and hence the category $\Ind\Coh(\LS_{\hat{B},\breve{F}}^\tame\times_{\LS_{\hat{G},\breve{F}}^\tame}\LS_{\hat{B},\breve{F}}^\tame)$ is $\cO(\hat{C})$-linear monoidal. The equivalence \eqref{eq: B_G tame} is naturally $\cO(\hat{C})$-linear as monoidal categories by Assumption \ref{assump: local geom Langlands}.(b). The $q$-power map $[q]$ is invertible on $\LS^\tame_{\hat{G},\breve{F}}$ and $\LS^\tame_{\hat{B},\breve{F}}$ with inverse $[q^{-1}]$.
    
    Under the fully faithful embedding (\cite[Theorem 2.86]{ZTame})
    \begin{equation}\label{eq: LS^tame tr}
        \mathrm{Tr}(\Ind\Coh(\LS_{\hat{B},\breve{F}}^\tame\times_{\LS_{\hat{G},\breve{F}}^\tame}\LS_{\hat{B},\breve{F}}^\tame),[q^{-1}]_*)\hookrightarrow \Ind\Coh(\LS_{\hat{G}}^\tame),
    \end{equation}
    the $\QCoh(\hat{C}^{[q]})$-action on the left defined by categorical trace is compatible with the $\QCoh(\hat{C}^{[q]})$-action on the right induced by \eqref{eq: LS -> C^[q]} by \cite[Proposition 8.81]{ZTame}. Recall that the functor $\LL_{G,\psi}^\tame$ is induced by the equivalence
    \[
    \mathrm{Tr}(\BB_{G,\psi}^\tame)\colon \mathrm{Tr}(\Shv_\mon(\mathrm{Iw}^u\backslash LG/\mathrm{Iw}^u,\overline{\FF}_\ell),\sigma_*)\simeq  \mathrm{Tr}(\Ind\Coh(\LS_{\hat{B},\breve{F}}^\tame\times_{\LS_{\hat{G},\breve{F}}^\tame}\LS_{\hat{B},\breve{F}}^\tame),[q^{-1}]_*)
    \]
    on categorical traces. Therefore $\LL_{G,\psi}^\tame$ is $\cO(\hat{C}^{[q]})$-linear. Now it follows that $\cF_{G,\psi}$ is $\cO(\hat{C}^{[q]})$-linear. The unipotent case follows by taking the unipotent connected component in $\hat{C}^{[q]}$.
\end{proof}

\subsection{\texorpdfstring{$t$}{t}-exactness of the finite Langlands functor}
Assume that $\ell\nmid |\pi_1(\hat{G})_\mathrm{tor}|$. Let $f\colon \LS^\tame_{\hat{G}}\to \BB\hat{G}$ denote the projection map. Let 
\begin{equation}\label{eq: f_* IndPerf}
    f_*^{\Ind\Perf}\colon \Ind\Perf(\LS^\tame_{\hat{G}})\to \Ind\Perf(\BB\hat{G})
\end{equation}
denote the continuous right adjoint of the functor $f^*\colon \Ind\Perf(\BB\hat{G})\to \Ind\Perf(\LS^\tame_{\hat{G}})$. By \cite[Theorem VIII.5.2]{FS21} and \cite[Theorem 4.7.4.5]{Lurie-HA}, there is a natural equivalence of categories
\[
    \Ind\Perf(\LS^\tame_{\hat{G}})\simeq \Mod_{f_*^{\Ind\Perf}\cO_{\LS^\tame_{\hat{G}}}}(\Ind\Perf(\BB \hat{G})).
\]
Recall that $\Ind\Perf(\BB\hat{G})$ admits two natural $t$-structures: the standard $t$-structure and the good filtration $t$-structure (\cite[Definition VIII.5.3]{FS21}). By \cite[Theorem VIII.5.2 and Corollary VIII.5.7]{FS21}, we know that $f_*^{\Ind\Perf}\cO_{\LS^\tame_{\hat{G}}}$ lies in the hearts of both $t$-structures. Indeed, $f_*^{\Ind\Perf}\cO_{\LS^\tame_{\hat{G}}}$ is clearly coconnective for the standard $t$-structure, and is connective for the good filtration $t$-structure by \emph{loc. cit.} Then the claim follows as the connective part of the good filtration $t$-structure is contained in the connective part of the standard $t$-structure. 

It follows that $\Ind\Perf(\LS^\tame_{\hat{G}})$ also carries a standard $t$-structure and a good filtration $t$-structure such that the functor \eqref{eq: f_* IndPerf} is $t$-exact for both $t$-structures.

Recall that there is a natural functor 
\begin{equation}\label{eq: Coh->IndPerf}
    \Coh(\LS^\tame_{\hat{G}})\hookrightarrow\Ind\Coh(\LS^\tame_{\hat{G}})\xrightarrow{\Psi}\Ind\Perf(\LS^\tame_{\hat{G}})
\end{equation}
where $\Psi$ is the continuous right adjoint of the fully faithful embedding 
\[\Psi^L\colon \Ind\Perf(\LS^\tame_{\hat{G}})\hookrightarrow \Ind\Coh(\LS^\tame_{\hat{G}})\] 
defined to be the ind-completion of $\Perf(\LS^\tame_{\hat{G}})\subseteq\Coh(\LS^\tame_{\hat{G}})$.

\begin{lemma}\label{lemma: Coh->IndPerf t-exact}
    Assume that $\ell\nmid |\pi_1(\hat{G})_\mathrm{tor}|$.
    The functor \eqref{eq: Coh->IndPerf} is fully faithful and $t$-exact when both sides are endowed with the standard $t$-structures.
\end{lemma}
\begin{proof}
    All the $t$-structures in this proof refer to the standard $t$-structures.
    We first show that \eqref{eq: Coh->IndPerf} is $t$-exact. Let $f\colon \LS^\tame_{\hat{G}}\to \BB\hat{G}$ denote the projection. By \cite[Lemma 9.26]{ZTame}, the functor
    \[
        f_*^{\Ind\Coh}\colon \Ind\Coh(\LS^\tame_{\hat{G}})\to \Ind\Coh(\BB\hat{G})\simeq\Ind\Perf(\BB\hat{G})
    \]
    is defined by ind-completing the functor
    \[
        \Coh(\LS^\tame_{\hat{G}})\subseteq \QCoh(\LS^\tame_{\hat{G}})^+\xrightarrow{f_*}\QCoh(\BB\hat{G})^+\simeq \Ind\Coh(\BB\hat{G})^+,
    \]
    which is $t$-exact as $f$ is affine. Here $(-)^+$ denotes the subcategory of bounded-below objects. Therefore $f_*^{\Ind\Coh}$ is $t$-exact. We have $f_*^{\Ind\Coh}= f_*^{\Ind\Perf}\circ\Psi$, and the functor $f_*^{\Ind\Perf}$ is conservative and $t$-exact. Therefore $\Psi\colon\Ind\Coh(\LS^\tame_{\hat{G}})\to \Ind\Perf(\LS^\tame_{\hat{G}})$ is $t$-exact.

    We claim that for any $n\geq 0$, the functor
    \[
        \Psi^{\geq -n}\colon \Ind\Coh(\LS^\tame_{\hat{G}})^{\geq -n}\to \Ind\Perf(\LS^\tame_{\hat{G}})^{\geq -n}
    \]
    is fully faithful. The adjunction
    \[
        f^*\colon \Ind\Perf(\BB\hat{G})^{\geq -n}\leftrightarrows\Ind\Perf(\LS^\tame_{\hat{G}})^{\geq -n}\colon f_*^{\Ind\Perf}
    \]
    is monadic by \cite[Theorem 4.7.4.5]{Lurie-HA} as $f_*$ commutes with colimits and is conservative. Therefore 
    \[
        \Ind\Perf(\LS^\tame_{\hat{G}})^{\geq -n}\simeq \Mod_{T}(\Ind\Perf(\BB\hat{G})^{\geq -n})\simeq \Mod_{T}(\QCoh(\BB\hat{G})^{\geq -n})\simeq \QCoh(\LS^\tame_{\hat{G}})^{\geq -n},
    \]
    where $T$ is the monad given by $f_*^{\Ind\Perf}\cO_{\LS^\tame_{\hat{G}}}$. We see that
    \[
        \Ind\Coh(\LS^\tame_{\hat{G}})^{\geq -n}\simeq \Ind\Perf(\LS^\tame_{\hat{G}})^{\geq -n}\simeq  \QCoh(\LS^\tame_{\hat{G}})^{\geq -n}
    \]
    by \cite[Lemma 9.21]{ZTame}. In particular, the functor \eqref{eq: Coh->IndPerf} is fully faithful.
\end{proof}

\begin{cor}\label{cor: F^fg linear}
    Suppose that Assumption \ref{assump: local geom Langlands} holds.
    The functor $\cF_{G,\psi}^{\unip,c}\colon \Rep^{\unip}_c(G(\FF_q),\overline{\FF}_\ell)\to \Coh(\LS^{\widehat\unip}_{\hat{G}})$ is $\cO(\hat{C}^{[q]}_u)$-linear.
\end{cor}
\begin{proof}
    By \cite[Theorem 5.4.(1)]{ZTame}, we have a chain of fully faithful continuous embeddings
    \[
        \Ind\Perf(\LS^{\widehat\unip}_{\hat{G}})\hookrightarrow \Shv^{\widehat\unip}(\Isoc_G,\overline{\FF}_\ell)\xhookrightarrow{\iota}\Ind\Shv^{\unip}_\fg(\Isoc_G,\overline{\FF}_\ell)\xhookrightarrow{\Ind\LL^{\unip,\fg}_{G,\psi}}\Ind\Coh(\LS^{\widehat\unip}_{\hat{G}})
    \]
    that preserve compact objects, where the second functor is the ind-completion of $\Shv^{\widehat\unip}(\Isoc_G,\overline{\FF}_\ell)^\omega\subseteq \Shv^\unip_\fg(\Isoc_G,\overline{\FF}_\ell)$. Passing to right adjoints, we see that the composition
    \begin{equation}\label{eq: 4 composition}
        \Ind\Shv_\fg^\unip(\Isoc_G,\overline{\FF}_\ell)\xrightarrow{\iota^R} \Shv^{\widehat\unip}(\Isoc_G,\overline{\FF}_\ell) \xrightarrow{\LL_{G,\psi}^{\widehat\unip}}\Ind\Coh(\LS^{\widehat\unip}_{\hat{G}})\xrightarrow{\Psi}\Ind\Perf(\LS^{\widehat\unip}_{\hat{G}})
    \end{equation}
    is identified with the composition
    \begin{equation}\label{eq: 3 composition}
        \Ind\Shv_\fg^\unip(\Isoc_G,\overline{\FF}_\ell)\xhookrightarrow{\Ind\LL_{G, \psi}^{\unip,\fg}}\Ind\Coh(\LS^{\widehat\unip}_{\hat{G}})\xrightarrow{\Psi}\Ind\Perf(\LS^{\widehat\unip}_{\hat{G}}).
    \end{equation}
    By Proposition \ref{prop: F_G linear}, the functor $\LL_{G,\psi}^{\widehat\unip}$ is $\QCoh(\hat{C}_u^{[q]})$-linear. As $\iota$ and $\Psi^L$ are $\QCoh(\hat{C}_u^{[q]})$-linear, the right adjoints $\iota^R$ and $\Psi$ are $\QCoh(\hat{C}_u^{[q]})$-linear by rigidity of $\QCoh(\hat{C}_u^{[q]})$. Therefore \eqref{eq: 4 composition} is $\QCoh(\hat{C}_u^{[q]})$-linear, and hence \eqref{eq: 3 composition} is $\QCoh(\hat{C}_u^{[q]})$-linear. In particular, the restriction of \eqref{eq: 3 composition} 
    \[
        \Shv_\fg^\unip(\Isoc_G,\overline{\FF}_\ell)\xhookrightarrow{\LL_{G,\psi}^{\unip,\fg}} \Coh(\LS^{\widehat\unip}_{\hat{G}})\xhookrightarrow{\Psi} \Ind\Perf(\LS^{\widehat\unip}_{\hat{G}})
    \]
    is $\Perf(C^{[q]}_u)$-linear, where the second functor is fully faithful by Lemma \ref{lemma: Coh->IndPerf t-exact}. Therefore $\LL^{\unip,\fg}_{G,\psi}$ is $\cO(\hat{C}^{[q]}_u)$-linear, and hence $\cF_{G,\psi}^{\unip,c}$ is $\cO(\hat{C}^{[q]}_u)$-linear.
\end{proof}

The following theorem is a generalization of \cite[Theorem 5.10]{ZTame} to modular coefficients. By \cite[Theorem 1.2]{BDR17}, if every elementary $\ell$-subgroup of $G(\FF_q)$ is contained in a torus, then \eqref{eq: u' in u} is an equivalence. In particular, the functor $\cF^{\unip,c}_{G,\psi}$ is defined on $\Rep_c^{\widehat\unip}(G(\FF_q),\overline{\FF}_\ell)$.

\begin{thm}\label{thm: t-exact CM}
    Suppose that Assumption \ref{assump: local geom Langlands} holds. Suppose every elementary $\ell$-subgroup of $G(\FF_q)$ is contained in a torus.  Then the following statements hold.
    \begin{enumerate}
        \item The functor 
        \[
            \cF_{G,\psi}^{\unip,c}\colon \Rep^{\widehat\unip}_c(G(\FF_q),\overline{\FF}_\ell)\to \Coh(\LS^{\widehat\unip}_{\hat{G}})
        \]
        is $t$-exact when both sides are endowed with the standard $t$-structures. 
        \item For any object $V\in \Rep_c^{\widehat\unip}(G(\FF_q),\overline{\FF}_\ell)^\heartsuit$, the associated coherent sheaf $\cF_{G,\psi}^{\unip,c}(V)$ is maximal Cohen--Macaulay.
    \end{enumerate}  
\end{thm}
\begin{proof}
    For $w\in W$ with a lift $\dot{w}\in G(\overline{\FF}_q)$, let 
    \[
        \widetilde{\Til}_{\dot{w},\hat{u}}^\mon \in\Shv_{u-\mon}(U\backslash G/U,\overline{\FF}_\ell)
    \]
    denote the indecomposable cofree unipotent monodromic tilting sheaf constructed in \cite[Proposition 4.50]{ZTame}. Let
    \[
        \Ch_{G,\phi}^{u-\mon}\colon \Shv_{u-\mon}(U\backslash G/U,\overline{\FF}_\ell)\to \Rep^{\widehat\unip}(G(\FF_q),\overline{\FF}_\ell)
    \]
    denote the Deligne--Lusztig induction functor in \cite[(4.62)]{ZTame}. By \cite[Theorem 4.91]{ZTame}, the objects
    \[
        \widetilde{R}^T_{\dot{w},\hat{u}}\coloneqq \Ch^{u-\mon}_{G,\phi}(\widetilde{\Til}^{\mon}_{\dot{w},\hat{u}}) 
    \]
    give a set of compact projective generators in $\Rep^{\widehat\unip}(G(\FF_q),\overline{\FF}_\ell)^\heartsuit$. Moreover, there is a non-canonical isomorphism
    \begin{equation}\label{eq: R^T self dual}
        \widetilde{R}^T_{\dot{w},\hat{u}}\simeq (\widetilde{R}^T_{\dot{w},\hat{u}})^\vee
    \end{equation}
    in $\Rep^{\widehat\unip}(G(\FF_q),\overline{\FF}_\ell)^\heartsuit$, as mentioned in the proof of \emph{loc. cit.} By \eqref{eq: D_GS F^c}, we see that
    \begin{equation}\label{eq: F(R^T) self dual}
        \DD_\GS'(\cF_{G,\psi}^{\widehat\unip}(\widetilde{R}^T_{\dot{w},\hat{u}}))\simeq \cF_{G,\psi^{-1}}^{\widehat\unip}(\widetilde{R}^T_{\dot{w},\hat{u}}).
    \end{equation}
    
    Let
    \[
        \Ch^{u-\mon}_{LG,\phi}\colon \Shv_{u-\mon}(\Iw^u\backslash LG/\Iw^u,\overline{\FF}_\ell)\to \Shv^{\widehat\unip}(\Isoc_G,\overline{\FF}_\ell)
    \]
    denote the affine Deligne--Lusztig induction functor in \cite[(4.46)]{ZTame}. By \eqref{eq: finite Hecke in affine Hecke}, we can view $\widetilde{\Til}_{\dot{w},\hat{u}}^\mon$ as a cofree unipotent monodromic tilting object in $\Shv_{u-\mon}(\Iw^u\backslash LG/\Iw^u,\overline{\FF}_\ell)$, and there is a natural isomorphism
    \[
        \Ch_{LG,\phi}^{u-\mon}(\widetilde{\Til}_{\dot{w}}^\mon)\simeq (i_1)_*\cInd_{G(O_F)}^{G(F)}\widetilde{R}^T_{\dot{w},\hat{u}}
    \]
    by \cite[Lemma 4.67]{ZTame}. Let $V$ be a tilting $\hat{G}$-representation. Let $\cZ^{\widehat\unip}(V)$ denote the unipotent monodromic central sheaf in \cite[Theorem 5.1.(2)]{ZTame}. By \cite[Theorem 4.136]{ZTame} and Assumption \ref{assump: local geom Langlands}.(c), we have
    \[
        R\Hom_{\Shv^{\widehat\unip}(\Isoc_G)}(\Ch^{u-\mon}_{LG,\phi}(\cZ^{\widehat\unip}(V)\star^u \widetilde{\Til}_{\dot{w},\hat{u}}^\mon),(i_1)_*\cInd_{G(O_F)}^{G(F)}\Gamma_\psi^{\widehat\unip})\in D(\overline{\FF}_\ell)^\heartsuit.
    \]
    By \cite[Lemma 5.6]{ZTame} and Theorem \ref{thm: cllc}.(2), this implies that
    \[
        R\Hom_{\Ind\Coh(\LS^{\widehat\unip}_{\hat{G}})}(\widetilde{V}\otimes \cF_{G,\psi}^{\widehat\unip}(\widetilde{R}^T_{\dot{w},\hat{u}}),\cO_{\LS^{\widehat\unip}_{\hat{G}}}) \in D(\overline{\FF}_\ell)^\heartsuit.
    \]
    Here $\widetilde{V}$ is the pullback of $V$ along $\LS^\tame_{\hat{G}}\to\BB\hat{G}$. By \eqref{eq: F(R^T) self dual}, we see that
    \[\begin{aligned}
        R\Hom_{\Ind\Coh(\LS^{\widehat\unip}_{\hat{G}})}(\widetilde{V}\otimes \cF_{G,\psi}^{\widehat\unip}(\widetilde{R}^T_{\dot{w},\hat{u}}),\cO_{\LS^{\widehat\unip}_{\hat{G}}}) &\simeq R\Hom_{\Ind\Coh(\LS^{\widehat\unip}_{\hat{G}})}(\DD_\GS'(\cO_{\LS^{\widehat\unip}_{\hat{G}}}),\DD_\GS'(\widetilde{V}\otimes \cF_{G,\psi}^{\widehat\unip}(\widetilde{R}^T_{\dot{w},\hat{u}})))\\
        &\simeq R\Hom_{\Ind\Coh(\LS^{\widehat\unip}_{\hat{G}})}(\cO_{\LS^{\widehat\unip}_{\hat{G}}}, c^*\widetilde{V}^\vee\otimes \cF_{G,\psi^{-1}}^{\widehat\unip}(\widetilde{R}^T_{\dot{w},\hat{u}}))\\
        &\simeq R\Hom_{\Ind\Coh(\LS^{\widehat\unip}_{\hat{G}})}(c^*\widetilde{V},\cF_{G,\psi^{-1}}^{\widehat\unip}(\widetilde{R}^T_{\dot{w},\hat{u}}))
    \end{aligned}\]
    is concentrated in degree 0. Here we use that $\cO_{\LS^{\widehat\unip}_{\hat{G}}}\simeq \omega_{\LS^{\widehat\unip}_{\hat{G}}}$ by \cite[Proposition 3.1.6]{Zhu2020Coherent}. This implies that $\Psi(\cF^{\widehat\unip}_{G,\psi}(\widetilde{R}^T_{\dot{w},\hat{u}}))\in \Ind\Perf(\LS^{\widehat\unip}_{\hat{G}})$ lies in the connective part of the good filtration $t$-structure. Indeed, by the highest weight structure, every Weyl module $\Delta_\mu$ for $\mu\in\XX^\bullet(\hat{T})^+$ admits a finite resolution of the form
    \[
        0\to \Delta_\mu\to V_0\to\cdots \to V_n\to 0
    \]
    with each $V_i$ tilting. Therefore if $M\in\Ind\Perf(\BB\hat{G})$ satisfies $R\Hom(V,M)\in D(\overline{\FF}_\ell)^{\leq 0}$ for any tilting representation $V$, then $R\Hom(\Delta_\mu,M)\in D(\overline{\FF}_\ell)^{\leq 0}$ for any $\mu\in \XX^\bullet(\hat{T})^+$, and hence $M$ lies in the connective part of the good filtration $t$-structure by definition. 
    
    In particular, we see that $\Psi(\cF_{G,\psi}^{\widehat\unip}(\widetilde{R}^T_{\dot{w},\hat{u}}))$ lies in the connective part of the standard $t$-structure. By Lemma \ref{lemma: Coh->IndPerf t-exact}, it follows that $\cF_{G,\psi}^{\widehat\unip}(\widetilde{R}^T_{\dot{w},\hat{u}})\in \Coh(\LS^{\widehat\unip}_{\hat{G}})$ lies in the connective part of the standard $t$-structure. Now by \eqref{eq: F(R^T) self dual} again, up to replacing $\psi$ by $\psi^{-1}$, we see that $\cF_{G,\psi}^{\widehat\unip}(\widetilde{R}^T_{\dot{w},\hat{u}})\in \Coh(\LS^{\widehat\unip}_{\hat{G}})^\heartsuit$ is concentrated in degree 0.

    Applying Lemma \ref{lem: ft representations are nice} to $\mathcal{A}=\Rep_c^{\widehat\unip}(G(\FF_q),\overline{\FF}_\ell)^\heartsuit$, $\bC=\Coh(\LS^{\widehat\unip}_{\ghat})$, $\cF=\cF_{G,\psi}^{\unip,c}$, and 
    \[
        P=\bigoplus_{w\in W}\widetilde{R}^T_{\dot{w},\hat{u}},
    \]
    we find that $\cF_{G,\psi}^{\unip,c}$ is $t$-exact. We note that $P$ is also an injective cogenerator by \eqref{eq: R^T self dual}, using that $(-)^\vee$ is an anti-equivalence of $\Rep_c^{\widehat\unip}(G(\FF_q),\overline{\FF}_\ell)^\heartsuit$. This proves (1). Now (2) follows from \eqref{eq: D_GS F^c} and \cite[Tag 0B5A]{stacks-project}.
\end{proof}

\begin{lemma}\label{lem: ft representations are nice}
    Let $\mathcal{A}$ be a finite-length abelian category with finitely many isomorphism classes of simple objects, and let $\bC$ be a stable category with a bounded $t$-structure.
    Let $\mathcal{F}: D^b(\mathcal{A}) \to \bC$ be an exact functor. Suppose that $P\in\mathcal{A}$ is both a projective generator and an injective cogenerator. If $\mathcal{F}(P)\in\bC^\heartsuit$, then the functor $\mathcal{F}$ is $t$-exact.
\end{lemma}
\begin{proof}
    By finiteness of simple objects, there are some integers $a\leq b$ such that $\mathcal{F}(L) \in \bC^{[a, b]}$ for every simple object $L$. Induction on length gives the same bound for $\cF(A)$ for every $A\in\mathcal{A}$. Choose a surjection $P^n \to A$ with kernel $K_1$. The exact sequence
    \[
        0 \to K_1 \to P^n \to A \to 0
    \]
    gives $\cH^i(\cF(A)) \simeq \cH^{i+1}(\cF(K_1))$ for $i > 0$, because $\cF(P)\in\bC^\heartsuit$. Iterating this process along a projective resolution of $K_i$, we find that $\cH^i(\cF(A)) \simeq \cH^{i + k}(\cF(K_k))$. Taking $k$ sufficiently large, we see that $\cH^i(\cF(A)) = 0$ for $i > 0$. Since $P$ is also an injective cogenerator, we can play the same game with injective resolutions and obtain that $\cH^i(\cF(A)) = 0$ for $i < 0$. Since the source $t$-structure is bounded, this finishes the proof. 
\end{proof}

\subsection{The case \texorpdfstring{$\ell\mid q-1$ with $\ell$}{l|q-1 with l} large}\label{subsection: unconditional}
For our applications to the commuting scheme, it suffices to know the statement of Theorem \ref{thm: t-exact CM} on the principal block $\rep^{0}_c(G(\FF_q),\overline{\FF}_\ell)$. On this block, the results of \cite{BR-modular-affine-Hecke,BR-modular-two-affine-Hecke} suffice; the full Assumption \ref{assump: local geom Langlands} is not needed.
To apply the results of \cite{BR-modular-two-affine-Hecke}, we make the following assumption in this subsection.

\begin{assump}\label{assump: large ell}
    \begin{enumerate}
        \item $\ell \mid q-1$.
        \item $\ell$ is larger than the Coxeter number of any simple factors of $G$, and $\ell\neq 19$ (resp. $\ell\neq 31$) when $G$ has a simple factor of type $E_7$ (resp. $E_8$).
        \item $G$ has connected center.
    \end{enumerate}
\end{assump}
In particular, $\ell$ does not divide the order of $W$. See Remark \ref{rmk: on assump local geom Langlands} for what is known in \cite{BR-modular-two-affine-Hecke}. In this case, the functors \eqref{eq: L hat unip} and \eqref{eq: L ufg} in Theorem \ref{thm: cllc} can still be defined. Therefore the functors \eqref{eq: F hatunip} and \eqref{eq: F unip c} are still available. Moreover, the unipotent part of Proposition \ref{prop: F_G linear} and Corollary \ref{cor: F^fg linear} are still true. Note that under Assumption \ref{assump: large ell}.(2), every elementary $\ell$-subgroup of $G(\FF_q)$ is contained in a torus. Therefore \eqref{eq: u' in u} is an equivalence.

Let $\LS^{\widehat\unip}_{\hat{B}}$ denote the (derived) moduli stack of unipotent Langlands parameters valued in $\hat{B}$ defined in \cite[\S 2.2]{ZTame}. Let $\LS^{\unip}_{\hat{B}}\subseteq \LS^{\widehat\unip}_{\hat{B}}$ denote the closed substack where the image of $\tau\in I_F$ lies in the unipotent radical $\hat{U}$. Let
\[
    \pi^\unip\colon \LS^{\unip}_{\hat{B}}\to \LS^{\widehat\unip}_{\hat{G}}\quad\text{resp.}\quad\pi^{\widehat\unip}\colon \LS^{\widehat\unip}_{\hat{B}}\to \LS^{\widehat\unip}_{\hat{G}}
\]
denote the natural projections. Define the coherent Springer sheaf
\[
    \CohSpr_{\hat{G}}^\unip\coloneqq (\pi^\unip)_*\omega_{\LS^{\unip}_{\hat{B}}}\quad \text{resp.}\quad\CohSpr_{\hat{G}}^{\widehat\unip}\coloneqq (\pi^{\widehat\unip})_*\omega_{\LS^{\widehat\unip}_{\hat{B}}}
\]
By \cite[Theorem 5.1.(2) and Theorem 5.2.(2)]{ZTame}, we know that
\[
    \cF^{\unip,c}_{G,\psi}(R_1)\simeq \CohSpr_{\hat{G}}^\unip\quad\text{resp.}\quad \cF^{\widehat\unip}_{G,\psi}(\widetilde{R}_1)\simeq \CohSpr_{\hat{G}}^{\widehat\unip}
\]
for $R_1,\widetilde{R}_1$ in Definition \ref{def: R_1}. Let $\Xi_{\hat{u}}$ denote the unipotent monodromic part of $\Xi$ in Definition \ref{def: big tilting}. Let $\Ch^{u-\mon}_{G,\phi}(\Xi_{\hat{u}})\in \Rep^{\widehat\unip}(G(\FF_q),\overline{\FF}_\ell)$ denote the Deligne--Lusztig induction.

\begin{lemma}\label{lemma: CohSpr selfdual}
    There are natural isomorphisms
    \[ 
        \DD_\GS(\CohSpr^\unip_{\hat{G}})\simeq \CohSpr^\unip_{\hat{G}}\quad\text{and}\quad\DD_\GS(\CohSpr^{\widehat\unip}_{\hat{G}})\simeq \CohSpr^{\widehat\unip}_{\hat{G}}.
    \]
\end{lemma}
\begin{proof}
    The morphisms $\pi^\unip$ and $\pi^{\widehat\unip}$ are representable and proper, and hence $*$-pushforwards along them commute with Grothendieck--Serre duality. By \cite[\S 2]{ZTame}, we can write $\LS^\unip_{\hat{B}}$ (resp. $\LS^{\widehat\unip}_{\hat{B}}$) as the derived $[q]$-fixed-point stack of $\hat{U}/\hat{B}$ (resp. $\hat{B}^\wedge_u/\hat{B}$). Therefore  $\LS^\unip_{\hat{B}}$ (resp. $\LS^{\widehat\unip}_{\hat{B}}$) is quasi-smooth with trivial dualizing sheaf, and therefore $\omega_{\LS^{\unip}_{\hat{B}}}\simeq\cO_{\LS^{\unip}_{\hat{B}}}$ (resp. $\omega_{\LS^{\widehat\unip}_{\hat{B}}}\simeq\cO_{\LS^{\widehat\unip}_{\hat{B}}}$) is canonically self-dual. The claim follows.
\end{proof}

\begin{lemma}\label{lemma: F_G big tilting}
    There is a natural isomorphism
    \[
        \cF_{G,\psi}^{\widehat\unip}(\Ch^{u-\mon}_{G,\phi}(\Xi_{\hat{u}}))\simeq \cO_{\LS^{\widehat\unip}_{\hat{G}}}\otimes_{\cO(\hat{C})}\cO(\hat{T}).
    \]
    In particular, $\cO_{\LS^{\widehat\unip}_{\hat{G}}}$ is a direct summand of $\cF_{G,\psi}^{\widehat\unip}(\Ch^{u-\mon}_{G,\phi}(\Xi_{\hat{u}}))$.
\end{lemma}
\begin{proof}
    Under the equivalence \eqref{eq: B_G hatunip}, we have
    \[
        \BB_{G,\psi}^{\widehat\unip}(\Xi_{\hat{u}})\simeq \omega_{\LS^{\widehat\unip}_{\hat{B},\breve{F}}\times_{\LS^{\widehat\unip}_{\hat{G},\breve{F}}}\LS^{\widehat\unip}_{\hat{B},\breve{F}}}.
    \]
    by \cite[Proposition 9.5]{BR-modular-two-affine-Hecke}. On the other hand, consider the map
    \[
        \widetilde\pi^{\widehat\unip}\colon \widetilde{\LS}^{\widehat\unip}_{\hat{G}}= \LS^{\widehat\unip}_{\hat{G}}\times^L_{\LS^{\widehat\unip}_{\hat{G},\breve{F}}}\LS^{\widehat\unip}_{\hat{B},\breve{F}}\to \LS^{\widehat\unip}_{\hat{G}}.
    \]
    Under the fully faithful embedding \eqref{eq: LS^tame tr}, the object $\omega_{\LS^{\widehat\unip}_{\hat{B}}\times_{\LS^{\widehat\unip}_{\hat{G}}}\LS^{\widehat\unip}_{\hat{B}}}$ is computed by $R(\widetilde\pi^{\widehat\unip})_*\omega_{\widetilde{\LS}^{\widehat\unip}_{\hat{B}}}$, and hence
    \[
        \cF^{\widehat\unip}_{G,\psi}(\Ch^{u-\mon}_{G,\phi}(\Xi_{\hat{u}}))\simeq (\widetilde\pi^{\widehat\unip})_*\omega_{\widetilde{\LS}^{\widehat\unip}_{\hat{B}}}
    \]
    After fixing a topological generator $\tau\in I_F$, we have a Cartesian diagram
    \[\begin{tikzcd}
        \widetilde{\LS}^{\widehat\unip}_{\hat{G}}\ar[r]\ar[d,"\widetilde{\pi}^{\widehat\unip}"swap] & \hat{B}^\wedge_u/\hat{B} \ar[d,"\pi_\GS^{\widehat\unip}"] \\
        \LS^{\widehat\unip}_{\hat{G}}\ar[r] & \hat{G}^\wedge_u/\hat{G}
    \end{tikzcd}\]
    where the horizontal maps are given by evaluation at $\tau$, and $\pi_\GS^{\widehat\unip}$ is the completion of the Grothendieck--Springer alteration along the unipotent cone. The first assertion follows from the well-known isomorphism $(\pi_\GS^{\widehat\unip})_*\cO_{\hat{B}^\wedge_u/\hat{B}}\simeq\cO_{\hat{G}^\wedge_u/\hat{G}} \otimes_{\cO(\that\git W)} \cO(\that)$. See, for example, \cite[Proof of Proposition 3.4.1]{BMR2008} for the Lie algebra version. The group version follows using the quasi-logarithms of \cite[Appendix C]{BKV}.

    The last assertion follows as $\cO(\hat{T})$ is finite free over $\cO(\hat{C})$ by \cite{Steinberg1975Pittie}.
\end{proof} 

The following result is conjectured in \cite[Conjecture 4.4.2]{Zhu2020Coherent}.

\begin{prop}\label{prop: CohSpr heart CM}
    The objects $\CohSpr^{\unip}_{\hat{G}}$ and $\CohSpr^{\widehat\unip}_{\hat{G}}$ lie in the heart of the standard $t$-structure of $\Coh(\LS^{\widehat\unip}_{\hat{G}})$ and are maximal Cohen--Macaulay.
\end{prop}
\begin{proof}
    Let $V$ be a tilting $\hat{G}$-representation. Let $\mathcal{Z}^{\widehat\unip}(V)$ denote the unipotent monodromic central sheaf in \cite[Theorem 7.8]{BR-modular-two-affine-Hecke}. Let $\widetilde\Delta_{e,\hat{u}}^{\mon}\in \Shv_{u-\mon}(\mathrm{Iw}^u\backslash LG/\mathrm{Iw}^u,\overline{\FF}_\ell)$ (\cite[\S 4.2.3]{ZTame}) denote the unit object.
    Applying \cite[Theorem 4.136]{ZTame} to $\cF=\widetilde\Delta_{e,\hat{u}}^{\mon}$ and $\mathcal{Z}=\mathcal{Z}^{\widehat\unip}(V)$, we see that
    \[
        R\Hom_{\Shv^{\widehat\unip}(\Isoc_G,\overline{\FF}_\ell)}(\Ch^{u-\mon}_{LG,\phi}(\mathcal{Z}^{\widehat\unip}(V)),(i_1)_*\cInd_{G(O_F)}^{G(F)}\Gamma_{\psi}^{\widehat\unip})\in D(\overline{\FF}_\ell)^\heartsuit.
    \]
    By \cite[Corollary 4.68 and Lemma 5.6]{ZTame}, we know that
    \[
        \LL_{G,\psi}^{\widehat\unip}((\Ch^{u-\mon}_{LG,\phi}(\mathcal{Z}^{\widehat\unip}(V)))\simeq \widetilde{V}\otimes\LL_{G,\psi}^{\widehat\unip}((\Ch^{u-\mon}_{LG,\phi}(\widetilde{\Delta}_{e,\hat{u}}^{\mon}))\simeq \widetilde{V}\otimes\cF^{\widehat\unip}_{G,\psi}(\widetilde{R}_1)\simeq\widetilde{V}\otimes\CohSpr_{\hat{G}}^{\widehat\unip}.
    \]
    Here $\widetilde{V}$ is the pullback of $V$ along $\LS^{\widehat\unip}_{\hat{G}}\to\BB\hat{G}$.
    Recall $\omega_{\LS^{\widehat\unip}_{\hat{G}}}\simeq\cO_{\LS^{\widehat\unip}_{\hat{G}}}$.
    By Lemma \ref{lemma: F_G big tilting} and Theorem \ref{thm: Ch big tilting}, we know that $\cO_{\LS^{\widehat\unip}_{\hat{G}}}$ is a direct summand of 
    \[
        \cF_{G,\psi}^{\widehat\unip}(\Ch^{u-\mon}_{G,\phi}(\Xi_{\hat{u}}))\simeq\cF_{G,\psi}^{\widehat\unip}(\Gamma_\psi^{\widehat\unip})\otimes_{\cO(\hat{C})}\cO(\hat{T}),
    \]
    which is a direct sum of copies of  $\cF_{G,\psi}^{\widehat\unip}(\Gamma_\psi^{\widehat\unip})$.
    It follows that
    \[
        R\Hom_{\Coh(\LS^{\widehat\unip}_{\hat{G}})}(\widetilde{V}\otimes\CohSpr_{\hat{G}}^{\widehat\unip},\cO_{\LS^{\widehat\unip}_{\hat{G}}})\in D(\overline{\FF}_\ell)^\heartsuit.
    \]
    By Lemma \ref{lemma: CohSpr selfdual}, the $\overline{\FF}_\ell$-complex
    \[\begin{aligned}
        R\Hom_{\Coh(\LS^{\widehat\unip}_{\hat{G}})}(\widetilde{V},\CohSpr_{\hat{G}}^{\widehat\unip}) &\simeq R\Hom_{\Coh(\LS^{\widehat\unip}_{\hat{G}})}(\DD_\GS(\CohSpr_{\hat{G}}^{\widehat\unip}),\DD_\GS(\widetilde{V})) \\
        &\simeq R\Hom_{\Coh(\LS^{\widehat\unip}_{\hat{G}})}(\CohSpr_{\hat{G}}^{\widehat\unip},\widetilde{V}^\vee)\\
        &\simeq R\Hom_{\Coh(\LS^{\widehat\unip}_{\hat{G}})}(\widetilde{V}\otimes\CohSpr_{\hat{G}}^{\widehat\unip},\cO_{\LS^{\widehat\unip}_{\hat{G}}})
    \end{aligned}\]
    is concentrated in degree $0$. Arguing as in the proof of Theorem \ref{thm: t-exact CM}, we see that $\CohSpr_{\hat{G}}^{\widehat\unip}$ is concentrated in degree $0$. The statement for $\CohSpr^\unip_{\hat{G}}$ follows as $\CohSpr_{\hat{G}}^{\widehat\unip}$ is a finite iterated extension of copies of $\CohSpr^\unip_{\hat{G}}$ by Lemma \ref{lemma: tilde R_1 inj proj}, and $\CohSpr^\unip_{\hat{G}}$ has bounded cohomology.

    Finally, the Cohen--Macaulayness of $\CohSpr^\unip_{\hat{G}}$ and $\CohSpr^{\widehat\unip}_{\hat{G}}$ follows from Lemma \ref{lemma: CohSpr selfdual} and \cite[Tag 0B5A]{stacks-project}.
\end{proof}

Let 
\[
    \cF_{G,\psi}^{0,c}\colon \Rep^0_c(G(\FF_q),\overline{\FF}_\ell)\to \Coh(\LS^{\widehat\unip}_{\hat{G}})
\]
denote the restriction of $\cF_{G,\psi}^{\unip,c}$ along \eqref{eq: 0 in unip}. Note that by Proposition \ref{prop:GGinRep0} and Theorem \ref{thm: Ch big tilting}, the object 
\[\Ch^{u-\mon}_{G,\phi}(\Xi_{\hat{u}})\simeq \Gamma_\psi^{\widehat\unip}\otimes_{\cO(\hat{C})}\cO(\hat{T})\]
lies in $\Rep^0_c(G(\FF_q),\overline{\FF}_\ell)^{\heartsuit}$.

We have the following theorem. 

\begin{thm}\label{thm: exact CM large ell}
    \begin{enumerate}
        \item The functor
    \[
    \cF_{G,\psi}^{0,c}\colon \Rep^0_c(G(\FF_q),\overline{\FF}_\ell)\to \Coh(\LS^{\widehat\unip}_{\hat{G}})
    \]
    is $t$-exact when both sides are endowed with the standard $t$-structures.
    \item For any $V\in \Rep^0_c(G(\FF_q),\overline{\FF}_\ell)^\heartsuit$, the associated coherent sheaf $\cF_{G,\psi}^{0,c}(V)$ is maximal Cohen--Macaulay.
    \end{enumerate} 
\end{thm}
\begin{proof}
    By Lemma \ref{lemma: R_1 ss}, simple representations $V$ in $\Rep^0_c(G(\FF_q),\overline{\FF}_\ell)^\heartsuit$ are direct summands of $R_1$. Therefore $\cF_{G,\psi}^{0,c}(V)$ is a direct summand of $\CohSpr_{\hat{G}}^\unip$, which is concentrated in degree $0$ and is maximal Cohen--Macaulay by Proposition \ref{prop: CohSpr heart CM}. The rest of the theorem is clear. 
\end{proof}

\section{Flatness}

Let $\hat{G}$ be a split reductive group over $\overline{\FF}_\ell$ with simply connected derived subgroup. Let
\[
    \Comm_{\hat{G}}\coloneqq\{x,y\in\hat{G}| xyx^{-1}=y\}
\]
denote the commuting scheme of $\hat{G}$. Let
\begin{equation}\label{eq: Comm -> C}
    \pi\colon \Comm_{\hat{G}}\to \hat{C}
\end{equation}
denote the map sending $(x,y)$ to the conjugacy class of $y$. In this section, we prove that the map \eqref{eq: Comm -> C} is flat at the identity point $u\in\hat{C}$.

Let $q\geq 1$ be an integer. Recall the $q$-commuting scheme
\[
    \Comm_{\hat{G}}^q=\{x,y\in\hat{G}|xyx^{-1}=y^q\}
\]
which admits a natural map
\begin{equation}\label{eq: Comm^q -> C^q}
    \Comm_{\hat{G}}^q\to \hat{C}^{[q]}
\end{equation}
sending $(x,y)$ to the image of $y$ in the Chevalley quotient.

Fix a closed immersion $\rho: \ghat \hookrightarrow \GL_N$ for some $N$. Write $\chi: \ghat \to \hat{C}$ for the (coarse) adjoint quotient, and let $\mathfrak{m}\subset\cO(\hat{C})$ be the maximal ideal of $u=\chi(1)$. For $a \in \mathbb{N}$, set 
\[\hat{C}_a\coloneqq \Spec\cO(\hat{C})/\mathfrak{m}^a.\] 
We let $C_N, \mathfrak{m}_N$ denote the analogous objects for $\GL_N$, and let $\alpha: \cO(C_N) \to \cO(\hat{C})$ denote the homomorphism induced by $\rho$.

\begin{lemma}\label{lemma: q=id}
    Let $a \leq \floor{\frac{\ell^{v_\ell(q-1)}}{N}}$. Then $[q] = \mathrm{Id}$ on $\ghat \times_{\hat{C}} \hat{C}_a$. 
\end{lemma}
\begin{proof}
    Let $R$ be a test algebra and let $g\in  (\ghat \times_{\hat{C}} \hat{C}_a)(R)$. Set $T = \rho(g) - \mathrm{Id}_N$. By Cayley--Hamilton,
    \[
        T^N + c_1 T^{N-1} \cdots + c_N=0
    \]
    with the coefficients $c_i \in \mathfrak{m}_N$.
    Thus $T^N \in \mathfrak{m}_N \cdot M_N(R)$, whence we arrive at $T^{Na} = 0$, because $\alpha(\mathfrak{m}_N) \subset \mathfrak{m}$ and $\mathfrak{m}^a=0$ on $\hat{C}_a$. Write $q-1=\ell^rb$ with $\ell\nmid b$. Then
    \[
        \rho(g)^{q-1}=(\mathrm{Id}_N+T)^{\ell^rb}=(\mathrm{Id}_N+T^{\ell^r})^{b}=\mathrm{Id}_N
    \]
    by the assumption $\ell^r\geq Na$. Since $\rho$ is a closed immersion, we have $g^q=g$ as required. 
\end{proof}

\begin{cor}\label{cor:Comm q vs Comm}
    Let $a \leq \floor{\frac{\ell^{v_\ell(q-1)}}{N}}$. Then there is a natural isomorphism
    \[
        \Comm_{\ghat} \times_{\hat{C}} \hat{C}_a \simeq \Comm_{\ghat}^q \times_{\hat{C}^{[q]}} \hat{C}_a.
    \]
\end{cor}
\begin{proof}
    This is clear by Lemma \ref{lemma: q=id}. 
\end{proof}

\begin{cor}\label{cor: T_a->C_a flat}
    Let $a \leq \floor{\frac{\ell^{v_\ell(q-1)}}{N}}$. Then $\hat{C}_a$ is a closed subscheme of $\hat{C}_u^{[q]}$. Furthermore one has
    \[
        \that \times_{\hat{C}} \hat{C}_a \simeq \hat{T}^{[q]}_u \times_{\hat{C}_u^{[q]}} \hat{C}_a.
    \]
    In particular $\that^{[q]}_u \times_{\hat{C}_u^{[q]}} \hat{C}_a \to \hat{C}_a$ is finite flat. 
\end{cor}
\begin{proof}
    By Lemma \ref{lemma: q=id}, we have $[q]=\mathrm{Id}$ on $\that \times_{\hat{C}} \hat{C}_a$. The claimed identification follows, and finite flatness follows by base change from $\hat T\to\hat C$.
\end{proof}

Let $q$ be a prime power with $\ell\mid q-1$ and let $\FF_q$ be the finite field. Let $G$ be the split reductive group over $\FF_q$ with dual group $\hat{G}$. Note that $G$ has connected center by assumption.

Recall the representation $\Gamma^{\widehat\unip}_\psi$ in \eqref{eq: unip GG}. Recall 
\[
    \End_{G(\FF_q)}(\Gamma_{\psi}^{\widehat\unip})\simeq \cO(\hat{C}^{[q]}_u).
\]
by \eqref{eq: End Gamma_psi^unip}. For $a \in \mathbb{N}$, we define $V_a \coloneqq \Gamma_{\psi}^{\widehat\unip} \otimes_{\cO(\hat{C}_u^{[q]})} \cO(\hat{C}_a)\in \Rep^{\widehat\unip}_c(G(\FF_q),\overline{\FF}_\ell)^\heartsuit$.

\begin{prop}\label{prop: V_a flat}    
    Assume that $\ell\nmid |W|$. If $a \leq \floor{\frac{\ell^{v_\ell(q-1)}}{N}}$, then $V_a$ is flat over $\cO(\hat{C}_a)$. 
\end{prop}
\begin{proof}
    Recall \eqref{eq: Hom R_e, Gamma} that
    \[
        \Hom_{G(\FF_q)}(\widetilde{R}_1, \Gamma_{\psi}^{\widehat\unip}) = \cO(\hat{T}_u^{[q]})
    \]
    and $\Gamma_{\psi}^{\widehat\unip}\in\Rep^0_c(G(\FF_q),\overline{\FF}_\ell)^{\heartsuit}$. By Proposition \ref{prop: principle block generator}, the object $\widetilde{R}_1$ is a projective generator of $\Rep_c^0(G(\FF_q),\overline{\FF}_\ell)$. Therefore $V_a$ is flat if and only if
    \[
        \Hom_{G(\FF_q)}(\widetilde{R}_1,V_a)=\Hom_{G(\FF_q)}(\widetilde{R}_1, \Gamma_{\psi}^{\widehat\unip}) \otimes_{\cO(\hat{C}^{[q]}_u)}\cO(\hat{C}_a)=\cO(\hat{T}^{[q]}_u\times_{\hat{C}^{[q]}_u}\hat{C}_a)
    \]
    is flat over $\cO(\hat{C}_a)$. This follows from Corollary \ref{cor: T_a->C_a flat}.
\end{proof}

\begin{cor}\label{cor: Comm^q C_a flat}
    Assume Assumption \ref{assump: large ell}.
    The map
    \begin{equation}\label{eq: Comm^q C_a}
        \Comm^q_{\hat{G}}\times_{\hat{C}^{[q]}}\hat{C}_a\to \hat{C}_a
    \end{equation}    
    induced by \eqref{eq: Comm^q -> C^q} is flat.
\end{cor}
\begin{proof}
    We need to show that
    \[
        \LS^{\widehat\unip}_{\hat{G}}\times_{\hat{C}^{[q]}_u}\hat{C}_a\to\hat{C}_a
    \]
    is flat. By Lemma \ref{lemma: F_G big tilting}, the structure sheaf $\cO_{\LS^{\widehat\unip}_{\hat{G}}}$ is a direct summand of $\cF_{G,\psi}^{0,c}(\Ch^{u-\mon}_{G,\phi}(\Xi_{\hat{u}}))$. It therefore suffices to show that
    \[
        \cF_{G,\psi}^{0,c}(\Ch^{u-\mon}_{G,\phi}(\Xi_{\hat{u}}))\otimes_{\cO(\hat{C}^{[q]}_u)}\cO(\hat{C}_a)
    \]
    is flat over $\cO(\hat{C}_a)$. By Corollary \ref{cor: F^fg linear} and Theorem \ref{thm: exact CM large ell}.(1), it suffices to show that
    \[
        \Ch^{u-\mon}_{G,\phi}(\Xi_{\hat{u}})\otimes_{\cO(\hat{C}^{[q]}_u)}\cO(\hat{C}_a)
    \]
    is flat over $\cO(\hat{C}_a).$ By Theorem \ref{thm: Ch big tilting}, using that $\cO(\hat{T})$ is finite free over $\cO(\hat{C})$, it suffices to show that $V_a$ is flat over $\cO(\hat{C}_a)$. This follows from Proposition \ref{prop: V_a flat}.
\end{proof}

\begin{cor}\label{cor:completion of comm is flat}
    Assume $\ell$ is larger than the Coxeter number of any simple factors of $\hat{G}$, and $\ell\neq 19$ (resp. $\ell\neq 31$) when $\hat{G}$ has a simple factor of type $E_7$ (resp. $E_8$). The map
    \[
        \pi\colon\Comm_{\hat{G}}\to \hat{C}
    \]
    is flat at $u\in \hat{C}$. In particular, the scheme $\Comm_{\hat{G}} \times_{\hat{C}} \hat{C}^{\wedge}_u$ is flat over $\hat{C}^{\wedge}_u$ at all of the points of $\pi^{-1}(u)$, where $\hat{C}^\wedge_u$ is the completion of $\hat{C}$ at $u$.
\end{cor}
\begin{proof}
    Fix $a\geq 0$. We can find a prime power $q$ such that $a<\floor{\frac{\ell^{v_\ell(q-1)}}{N}}$. Then by Corollary \ref{cor:Comm q vs Comm} and Corollary \ref{cor: Comm^q C_a flat}, the map
    \[
        \Comm_{\hat{G}}\times_{\hat{C}}\hat{C}_a\to \hat{C}_a
    \]
    is flat. Since $a$ is arbitrary, we find that $\Comm_{\hat{G}} \times_{\hat{C}} \hat{C}^\wedge_u \to \hat{C}^\wedge_u$ is flat. As the map $\hat{C}^\wedge_u\to\hat{C}_{(u)}$ to the localization at $u$ is faithfully flat, the map $\pi$ is flat at $u$.
\end{proof}

\section{Cohen--Macaulayness}\label{section: CM}

Let $\hat{G}$ be a split reductive group over $\overline{\FF}_\ell$ and let $G$ be its Langlands dual group. We assume that $\ell$ is larger than the Coxeter number of any simple factors of $\hat{G}$, and $\ell\neq 19$ (resp. $\ell\neq 31$) when $G$ has a simple factor of type $E_7$ (resp. $E_8$) throughout this section.

\subsection{Cohen--Macaulayness of the commuting scheme}

\newcommand{\Commu}{\Comm^u}
\begin{defn}
    We denote by $\Commu_{\ghat}$ the unipotent commuting scheme, defined to be the fiber product
    \[
        \begin{tikzcd}
        & \Commu_{\ghat} \arrow{r} \arrow{d}&\Comm_{\ghat}\arrow[d, "\pi"]\\
        & \Spec \flbar \arrow[r, "u"]& \hat{C}
        \end{tikzcd}
    \]
    where $\pi$ is \eqref{eq: Comm -> C}. We can write
    \[
        \Comm_{\hat{G}}^u=\{x\in \hat{G},y\in\cU_{\hat{G}}\,|\, xyx^{-1}=y\}
    \]
    where $\cU_{\hat{G}}=\hat{G}\times_{\hat{C}}\{u\}$ is the unipotent cone in $\hat{G}$.
\end{defn}

\begin{prop}\label{prop:unipotent LS is CM}
    Assume that the derived subgroup of $\hat{G}$ is simply connected. Then the scheme $\Comm^u_{\hat{G}}$ is Cohen--Macaulay.
\end{prop}
\begin{proof}
    We fix a prime power $q$ such that $\floor{\frac{\ell^{v_\ell(q-1)}}{N}}\geq 1$. Let $G$ be the split reductive group over $\FF_q$ with dual group $\hat{G}$. By Corollary \ref{cor:Comm q vs Comm}, we know that
    \[
        \Comm^q_{\hat{G}}\times_{\hat{C}^{[q]}}\{u\}\simeq \Comm^u_{\hat{G}}.
    \]
    As in the proof of Corollary \ref{cor: Comm^q C_a flat}, the sheaf
    \[
        \cO_{\Comm_{\hat{G}}^u/\hat{G}}\simeq \cO_{\LS^{\widehat\unip}_{\hat{G}}}\otimes_{\cO(\hat{C}^{[q]}_u)}\overline{\FF}_\ell
    \]
    is a direct summand of 
    \[\cF_{G,\psi}^{0,c}(\Ch^{u-\mon}_{G,\phi}(\Xi_{\hat{u}}))\otimes_{\cO(\hat{C}^{[q]}_u)}\overline{\FF}_\ell\simeq \cF_{G,\psi}^{0,c}(\Ch^{u-\mon}_{G,\phi}(\Xi_{\hat{u}})\otimes_{\cO(\hat{C}^{[q]}_u)}\overline{\FF}_\ell),\]
    and therefore is Cohen--Macaulay by Theorem \ref{thm: exact CM large ell}.(2). Therefore $\Comm_{\hat{G}}^u$ is Cohen--Macaulay.
\end{proof}

We next remove the assumption that the derived subgroup is simply connected by descent along an \'etale central isogeny. 

\begin{lemma}\label{lemma: Comm^u isog}
    Let $\hat{G} \to \hat{G}'$ be an \'etale central isogeny of reductive algebraic groups over a field $\overline{\FF}_\ell$ such that $\ell \nmid |W|$, then the natural map $\Commu_{\hat{G}} \to \hat{G} \times_{\hat{G}'} \Commu_{\hat{G}'}$ is an isomorphism.
\end{lemma}
\begin{proof}
    By \cite[Lemma 4.10]{Springer}, the map on unipotent cones $\cU_{\hat{G}} \to \cU_{\hat{G}'}$ is an isomorphism. Then the statement is clear.
\end{proof}

In particular, $\Commu_{\hat{G}}$ is Cohen--Macaulay if and only if the same is true for $\Commu_{\hat{G}'}$. 

\begin{cor}\label{cor: Comm u CM}
    The unipotent commuting scheme $\Commu_{\hat{G}}$ is Cohen--Macaulay. 
\end{cor}
\begin{proof}
    By assumption on $\ell$, the map $\widetilde{\hat{G}}=\hat{G}_{\mathrm{sc}}\times Z(\hat{G})^\circ\to \hat{G}$ is an \'etale central isogeny, where $\hat{G}_{\mathrm{sc}}$ is the simply connected cover of $\hat{G}_\mathrm{der}$. By Proposition \ref{prop:unipotent LS is CM}, the scheme $\Commu_{\widetilde{\hat{G}}}$ is Cohen--Macaulay. Therefore $\Commu_{\hat{G}}$ is Cohen--Macaulay by Lemma \ref{lemma: Comm^u isog}.
\end{proof}

\begin{thm}\label{thm:CM}
    The commuting scheme $\Comm_{\hat{G}}$ is Cohen--Macaulay.
\end{thm}
\begin{proof}
    The proof proceeds, roughly speaking, by a standard reduction to the completion of the unipotent cone. First let $(g, h) \in \Comm_{\ghat}(k)$ for an algebraically closed field $k$ over $\overline{\FF}_\ell$, then by the Jordan decomposition we may write
    \[
        g = su, h = tv
    \]
    where $s, t$ are semisimple and $u, v$ are unipotent. We define $\hat{L} = Z_{\hat{G}}(t)$. By \cite[Corollary 9.4]{Steinberg1968}, $\hat{L}$ is a possibly disconnected reductive group. We claim that there is a natural map 
    \[
        \Comm^{\wedge}_{\hat{G}, (g, h)} \to \mathbb{A}^{\dim(\hat{G}) - \dim(\hat{L}), \wedge}_0 \times \Comm^{\wedge}_{\hat{L}, (t^{-1}g, v)},
    \]
    which will turn out to be an isomorphism. Indeed if $\hat{\mathfrak{l}} = \lie(\hat{L})$ then we may choose a local smooth transversal $T$ to $\hat{L}$ at the identity, such that on tangent spaces $T_1(T) \oplus \hat{\mathfrak{l}} = \hat{\mathfrak{g}}$ and $T_1(T)$ is the sum of the non-trivial eigenspaces of $t$ on $\hat{\mathfrak{g}}$. Then the conjugation action of $T$ on $\hat{L}$
    \[
        \operatorname{Ad}_T\colon T \times \hat{L} \to \hat{G}
    \]
    induces an isomorphism $T^{\wedge}_1 \times \hat{L}_h^{\wedge} \to \hat{G}^{\wedge}_h$. Now taking $R$ an Artinian local $k$-algebra we can take $(a, b) \in \Comm_{\hat{G}}(R)$ lifting $(g, h)$. Using the above isomorphism we may uniquely write $b = xlx^{-1}$ for $l \in \hat{L}(R), x \in T(R)$, where $l$ reduces to $h$ and $x$ to $1$ modulo the maximal ideal of $R$. Then if we conjugate by $x$ we get a commuting pair $(x^{-1}ax, l)$, and we wish to show that $a_0 = x^{-1}ax \in \hat{L}(R)$. But again by inspection of the map on tangent spaces, the map $m|_{T \times \hat{L}}\colon T \times \hat{L} \to \hat{G}$ induced by multiplication gives a set of coordinates on $\hat{G}^{\wedge}_g$ around $g$, so we find that $a_0 = yl'$ for some $l' \in \hat{L}(R)$ reducing to $g$ modulo the maximal ideal, and $y \in T(R)$ reducing to the identity. Then $l'll'^{-1} = y^{-1}ly$, but since the left-hand side is in $\hat{L}(R)$, we find that $y$ must be trivial. The map $(a, b) \mapsto (x, x^{-1}ax, l)$ gives an isomorphism 
    \[
        \Comm^{\wedge}_{\hat{G}, (g, h)} \to \mathbb{A}^{\dim(\hat{G}) - \dim(\hat{L}), \wedge}_0 \times \Comm^{\wedge}_{\hat{L}, (g, h)},
    \]
    and since $t \in \hat{L}(k)$ is central we have a further isomorphism 
    \[
        \Comm^{\wedge}_{\hat{L}, (g, h)} \to \Comm^{\wedge}_{\hat{L}, (t^{-1}g, v)}.
    \]
    Let $\hat{M} = Z_{\hat{L}}(s)$, where $s$ is the semisimple part of $g$. Then $\hat{M}$ is a possibly disconnected reductive group again by \cite[Corollary 9.4]{Steinberg1968}. Applying the same reduction to the first coordinate gives 
    \[
        \Comm_{\hat{G}, (g, h)}^{\wedge} \simeq \mathbb{A}^{\dim(\hat{G}) - \dim(\hat{M}), \wedge}_0 \times \Comm_{\hat{M}, (u, v)}^{\wedge} \simeq \mathbb{A}^{\dim(\hat{G}) - \dim(\hat{M}), \wedge}_0 \times \Comm^{\wedge}_{\hat{M}^\circ, (u, v)}
    \]
    where the last isomorphism follows since now $u, v$ are both unipotent. We thus reduce to showing that if $\hat{G}$ is connected reductive and $(u, v)$ are unipotent elements of $\hat{G}$, then $\Comm_{\hat{G}, (u, v)}^{\wedge}$ is Cohen--Macaulay. Now we note that since $\ell\nmid |W|$ (and this condition passes along our reduction procedure to connected components of centralizers of semisimple elements) in particular $\ell \nmid |\pi_1(\hat{G}_{\mathrm{der}})|$, thus $\hat{T}\git W$ is smooth at the identity by \cite[Lemma 4.2]{Springer}. Finally take $\widetilde{\hat{G}} \to\hat{G}$ an \'etale cover with simply connected derived group, and denote by $Z$ the kernel of this isogeny. The map $\widetilde{\hat{T}}\git W \to \hat{T}\git W$ is smooth at the identity, and the base change
    \[
        \begin{tikzcd}
            &X \arrow{r}\arrow{d} & \widetilde{\hat{G}}^2 \arrow{d}\\
            &\Comm_{\hat{G}} \arrow{r} & \hat{G}^2
        \end{tikzcd}
    \]
    parameterizes pairs of elements $(a, b) \in \widetilde{\hat{G}}^2$ with commutator inside of $Z$. Since $Z$ is a finite \'etale group scheme, $\Comm_{\widetilde{\hat{G}}}$ is closed and open in $X$, and the fiber of $\Comm_{\widetilde{\hat{G}}}$ over the identity in $\widetilde{\hat{T}}\git W$ maps surjectively onto the analogous fiber for $\Comm_{\hat{G}}$ by Lemma \ref{lemma: Comm^u isog}, so we find that $\Comm_{\widetilde{\hat{G}}} \to \Comm_{\hat{G}}$ is finite \'etale surjective in a neighborhood of the unipotent fiber. Thus the map $\Comm_{\hat{G}} \to \hat{T}\git W$ is flat at the identity. Cohen--Macaulayness of $\Comm_{\hat{G}, (u, v)}^{\wedge}$ now follows from the fact that $\Commu_{\hat{G}}$ is Cohen--Macaulay by Corollary \ref{cor: Comm u CM} and \cite[Tag 045J]{stacks-project}. 
\end{proof}

\begin{cor}\label{cor: flatness of pi}
    Suppose that the derived group of $\ghat$ is simply connected. Then the morphism
    \[
        \pi: \Comm_{\ghat} \to \hat{C}
    \]
    is flat.
\end{cor}
\begin{proof}
    It is easy to see that the fibers all have dimension $\dim(\ghat)$ by decomposing each fiber of the map $\ghat \to \hat{C}$ into a finite union of orbits and using the orbit-stabilizer dimension formula. By \cite[Theorem 6.1]{Steinberg1965}, the map $\hat{C} \to \Spec\overline{\FF}_\ell$ is smooth if $\hat{G}$ has simply connected derived subgroup. The result then follows from Theorem \ref{thm:CM} and ``miracle flatness'' \cite[Tag 00R4]{stacks-project}.
\end{proof}

We now combine Cohen--Macaulayness with regularity in codimension one to establish normality of the commuting scheme.

\begin{cor}\label{cor: normal}
    The commuting scheme $\Comm_{\ghat}$ is normal. In particular, $\Comm_{\ghat}$ is reduced.
\end{cor}
\begin{proof}
    This is a consequence of Serre's criterion (\cite[Tag 031S]{stacks-project}), provided we can show that $\Comm_{\ghat}$ is regular in codimension $1$. First we prove this under the assumption that $\ghat$ has simply connected derived group, then we will reduce to this case. First we show that the locus of $(g, h) \in \Comm_{\ghat}$ such that either $g$ or $h$ is regular is smooth. By symmetry we prove this for the locus for which $g$ is regular. By \cite[Theorem 3.11]{Springer}, the locus of pairs $(g, h) \in \Comm_{\ghat}$ such that $g$ is regular is smooth after taking the quotient by the action of $Z(\ghat)$ on the second factor. Since $Z(\hat{G})$ is smooth, we find that this locus is smooth. 

    By Corollary \ref{cor: flatness of pi}, if $\eta \in \Comm_{\ghat}$ is a point of height one, then writing $\pi(\eta) = \eta'$ we have, by flatness, that $\eta'$ is either of height one or zero. In the latter case we are done as then $\eta'$ is regular, so we focus on the height one case. In this case $\eta$ is a generic point of the fiber $\Comm_{\ghat, \eta'}$, and its image $\eta'$ is a generic point of the discriminant locus. Let $s \in \ghat(\overline{k(\eta')})$ be a semisimple representative of the point $\eta'$, then $Z_{\ghat}(s)$ has semisimple rank $1$, and is connected by \cite[Corollary 9.4]{Steinberg1968}. Let $(g, h) \in \Comm_{\ghat, \eta'}$, then $h$ is conjugate to $su$ with $s$ the generic semisimple element as above and $u \in Z_{\ghat}(s)$, if $u \neq 1$, then $u$ is regular since $Z_{\ghat}(s)$ has semisimple rank one, thus $h$ is regular. Otherwise $h = s$, $g$ can be any element of a maximal torus of $\ghat$, so $Z_{\ghat}(s) \cap \ghat^{\mathrm{reg}}$ is dense and open (as it is nonempty). In particular the generic point of $\Comm_{\ghat, \eta'}$ is in the regular locus. 

    Now we move to the general case. By \cite[Tag 0FIZ]{stacks-project} it suffices to check normality of $\Comm_{\ghat}$ after taking completion at an arbitrary point $(g, h) \in \Comm_{\ghat}$. The same formally smooth reduction as applied in the proof of Theorem \ref{thm:CM} shows that we may reduce to studying the completion at a unipotent tuple $(u, v) \in \Comm_H$ for $H$ some connected reductive subgroup of $\ghat$. Now as usual we may choose $\alpha: \widetilde{H} \to H$ an \'etale central isogeny such that $\widetilde{H}_{\mathrm{der}}$ is simply connected, and the map 
    \[
        \Comm_{\widetilde{H}, (u, v)}^{\wedge} \to \Comm_{H, (u, v)}^{\wedge}
    \]
    is an isomorphism as argued in the last paragraph of the proof of Theorem \ref{thm:CM}. The scheme $\Comm_{\widetilde{H}, (u, v)}^{\wedge}$ is normal by \cite[Tag 0C23]{stacks-project}, and the same is true for $\Comm_{H, (u, v)}^{\wedge}$, completing the proof.
\end{proof}
\newcommand{\OF}{\mathcal{O}_F}
\newcommand{\OFS}{\OF[S^{-1}]}

\begin{cor}\label{cor: char 0}
    Let $\hat{G}/F$ be a connected reductive algebraic group over a field $F$ of characteristic $0$ or of characteristic $\ell$ satisfying the standing hypotheses. Then the commuting scheme $\Comm_{\hat{G}}$ is normal and Cohen--Macaulay.
\end{cor}
\begin{proof}
    When $F$ is of characteristic $\ell$, it follows from \cite[Tag 045P and Tag 038O]{stacks-project} and Theorem \ref{thm:CM} and Corollary \ref{cor: normal}. Now suppose that $F$ has characteristic $0$. We first assume that $F$ is a number field.
    Let $\OF$ denote the ring of integers in $F$. One can spread out $\hat{G}$ to a connected reductive group scheme $\hat{\mathcal{G}}/\OF[S^{-1}]$, for $S$ a finite collection of places of $F$ such that every place not in $S$ satisfies the assumption on characteristic. By openness of the flat locus, after enlarging $S$ one may assume that $\Comm_{\hat{\mathcal{G}}}$ is flat over $\OFS$. Let $v \in \Spec \OFS$ be a closed point with residue field $k(v)$. By the above discussion we know that $\Comm_{\hat{\mathcal{G}},k(v)}$ is Cohen--Macaulay. By \cite[Tag 045J]{stacks-project}, this implies that $\Comm_{\hat{\mathcal{G}}}$ is Cohen--Macaulay at all of its local rings in the fiber above $v$. However as $v$ was arbitrary, this implies the result: indeed take $(g, h) \in \Comm_{\hat{G}, F}$, since $\hat{\mathcal{G}} \times \hat{\mathcal{G}}$ is a finitely presented affine scheme over $\OFS$, there exists some place $v \in \Spec \OFS$ such that the coordinates of $(g, h)$ are regular at $v$. Thus $\Comm_{\hat{G}}$ is Cohen--Macaulay at $(g, h)$. For normality, note that once we know $\Comm_{\hat{G}}$ is Cohen--Macaulay, the same proof as in Corollary \ref{cor: normal} works without modification. For a general field $F$ of characteristic $0$, it follows from  \cite[Tag 045P and Tag 038O]{stacks-project} again.
\end{proof}

\subsection{The Lie algebra setting}

\newcommand{\lieghat}{\hat{\mathfrak{g}}}
\newcommand{\liethat}{\hat{\mathfrak{t}}}

Now let $\lieghat$ denote the Lie algebra of $\ghat$ over $\flbar$. Let $\liethat$ denote the Lie algebra of the maximal torus $\that$ of $\ghat$. Let
\[
    \Comm_{\lieghat}=\{X,Y\in \lieghat\,|\, [X,Y]=0\}
\]
be the commuting scheme for $\lieghat$. There is a natural morphism
\[
    \pi\colon \Comm_{\lieghat}\to \liethat\git W
\]
similar to \eqref{eq: Comm -> C}.

We recall the following definition.
\begin{defn}
    A map $L: \ghat \to \lieghat$ is called a quasi-logarithm if 
    \begin{enumerate}
        \item We have $L(1) = 0$,
        \item The morphism $L$ is $\operatorname{Ad}(\ghat)$-equivariant,
        \item The induced map on tangent spaces $dL_1$ is the identity. 
    \end{enumerate}
\end{defn}

By \cite[Appendix C]{BKV} such a quasi-logarithm exists under our hypotheses if $\ell$ is assumed very good for $\ghat$ and $\ghat$ is split with simply connected derived subgroup. As in this section we are in the business of studying Lie algebras, this last restriction on $\pi_1(\ghat_{\mathrm{der}})$ is not restrictive. 

For the remainder of this subsection we assume $\ell\nmid |W|$, and that $\ghat$ has simply connected derived subgroup. In particular, $\ghat$ admits a quasi-logarithm.

\begin{thm}\label{thm: Lie CM}
    The morphism $\pi: \Comm_{\lieghat} \to \liethat\git W$ is flat, and the scheme $\Comm_{\lieghat}$ is Cohen--Macaulay. 
\end{thm}
\begin{proof}
    Let $L: \ghat \to \lieghat$ be a quasi-logarithm. There is a map
    \begin{equation}\label{eq: completion lie vs group}
        \Comm^{\wedge}_{\ghat, (1, 1)} \to \Comm^{\wedge}_{\lieghat, (0, 0)}
    \end{equation}
    given by $(L, L)$. We claim that this map is an isomorphism. By $\mathrm{Ad}(\ghat)$-equivariance of the quasi-logarithm we know that this map lives over an identification $(\that\git W)^{\wedge}_u \simeq (\liethat\git W)^{\wedge}_0$ induced by $L$. Note that since the map $\hat{L}: \ghat^{\wedge}_1 \to \lieghat^{\wedge}_0$ induced by $L$ is an isomorphism since it is the identity on tangent spaces, we simply have to check that the commuting equations are well-behaved, and this follows from $\mathrm{Ad}(\ghat)$-equivariance once again. 

    By Corollary \ref{cor:completion of comm is flat} and Theorem \ref{thm:CM}, we find that $\Comm^{\wedge}_{\lieghat, (0, 0)} \to (\liethat \git W)^{\wedge}_0$ is flat and the source is Cohen--Macaulay. Now we make use of the natural scaling action of $\mathbb{G}_m$ on $\Comm_{\lieghat}$. Indeed for a pair $A, B \in \lieghat$ such that $[A, B]_{\lieghat} = 0$ we can scale both matrices by $t$. If $\chi: \lieghat \to \liethat\git W$ is the map to the adjoint quotient, then $\chi$ obeys the equivariance $\chi(tB) = (t^{d_1} f_1(B), t^{d_2}f_2(B), \dots, t^{d_m} f_m(B)),t\in\GG_m$ where the $d_i$ are the degrees of homogeneous generators of $\cO(\lieghat)^{\ghat}$, and all these weights are positive. We thus reduce to the situation of the following lemma, which completes the proof.
\end{proof}

\begin{lemma}\label{lem: dilation}
    Let $R \to A$ be a map of positively graded finitely presented $k$-algebras, where $R = k[t_1, \dots, t_n]$, with $R_0 = k,  A_0 = k'$ with $k'$ an extension of $k$. Let $I$ be $(t_1, \dots, t_n)$, $\mathfrak{m} = A_+$, and suppose that $\hat{A}_{\mathfrak{m}}/I\hat{A}_{\mathfrak{m}}$ is Cohen--Macaulay and $\hat{A}_{\mathfrak{m}}$ is flat over $\hat{R}$, then $A$ is flat and Cohen--Macaulay over $R$. 
\end{lemma}
\begin{proof}
    Since $\hat{A}_{\mathfrak{m}}$ is flat over $\hat{R}$, the Cohen--Macaulayness of $\hat{A}_{\mathfrak{m}}/I\hat{A}_{\mathfrak{m}}$ implies that of $\hat{A}_{\mathfrak{m}}$ by \cite[Tag 045J]{stacks-project}. By \cite[Tag 07NX]{stacks-project} this is then also true for $A_{\mathfrak{m}}$. Let $Z$ be the locus on which $\Spec A$ is not Cohen--Macaulay, by \cite[Tag 045U]{stacks-project} this is a closed subscheme of $\Spec A$, which is $\mathbb{G}_m$-invariant. Let $J$ be the homogeneous ideal defining this locus $Z$. If $J$ were a \emph{proper} homogeneous ideal it would be contained within $\mathfrak{m}$, but this is impossible by the above. The argument for flatness is essentially the same. 
\end{proof}

\begin{cor}
    The commuting scheme $\Comm_{\lieghat}$ is normal.
\end{cor}
\begin{proof}
    By Corollary \ref{cor: normal}, the completion $\Comm_{\hat{G},(1,1)}^\wedge$ is normal and integral. Therefore $\Comm_{\lieghat,(0,0)}^\wedge$ is normal and integral by the isomorphism \eqref{eq: completion lie vs group}.
    Recall that $\cO(\Comm_{\lieghat})$ is a positively graded ring, and hence the homomorphism $\cO(\Comm_{\lieghat})\to \cO(\Comm_{\lieghat,(0,0)}^\wedge)$ is injective. In particular, $\Comm_{\lieghat}$ is integral and is normal at the point $(0,0)$ by \cite[Tag 0FIZ]{stacks-project}. By \cite[Tag 035S]{stacks-project}, since $\Comm_{\lieghat}$ is Nagata, the normal locus in$\Comm_{\lieghat}$ is open dense. Let $Z$ be the closed locus on which the map from the normalization is not an isomorphism. Once again $Z$ is defined by a homogeneous ideal $J$ in the coordinate ring $A = \mathcal{O}(\Comm_{\lieghat})$. Now applying the argument of Lemma \ref{lem: dilation} we find that $\Comm_{\lieghat}$ is normal. 
\end{proof}

\begin{cor}\label{cor: char 0 Lie}
    Let $\hat{G}$ be a reductive algebraic group over a field $F$ of characteristic $0$ or of characteristic $\ell$ satisfying the standing hypotheses. If $\hat{\mathfrak{g}}$ is the Lie algebra of $\hat{G}$, then the scheme $\Comm_{\hat{\mathfrak{g}}}$ is normal and Cohen--Macaulay.
\end{cor}
\begin{proof}
    The proof is exactly the same as in corollary \ref{cor: char 0}.
\end{proof}

\appendix
\section{Trace of the big tilting sheaf}
Let $G$ be a split reductive group over a finite field $\FF_q$ with connected center. Let $\Lambda=\overline{\FF}_\ell$ or $\overline{\QQ}_\ell$ with $\ell$ good for $G$. Let $(B,T,e\colon U^{-}\to\GG_a)$ be a pinning of $G$. Let
\[\psi\colon U^-(\FF_q)\xrightarrow{e} \FF_q\to  \Lambda^\times\]
be a non-degenerate character. Recall the Gelfand--Graev representation
\[\Gamma_\psi\coloneqq \Ind_{U^-(\FF_q)}^{G(\FF_q)}\psi\in \Rep(G(\FF_q),\Lambda).\]

Let $(\Shv_\mon(U\backslash G/U,\Lambda),\star)$ be the monodromic Hecke category as in \cite[\S 4]{ZTame}. Consider the horocycle correspondence
\begin{equation}\label{eq: horocycle finite}
    U\backslash G/U\xleftarrow{p} G/\Ad_\sigma U\xrightarrow{q} G/\Ad_\sigma G=\BB G(\FF_q).
\end{equation}
Here $\Ad_\sigma(g)(h)=gh\sigma(g)^{-1}$. Let 
\[
F=\sigma_*\colon \Shv_\mon(U\backslash G/U,\Lambda)\xrightarrow{\simeq }\Shv_\mon(U\backslash G/U,\Lambda)
\]
denote the Frobenius automorphism.

Recall the following theorem on categorical traces.

\begin{thm}[{\cite[Theorem 4.97]{ZTame},\cite[Theorem 1.1.1]{Eteve2024FreeMonodromic}}]
    There is a natural equivalence
    \[\Tr(\Shv_\mon(U\backslash G/U,\Lambda),F)\simeq \Rep(G(\FF_q),\Lambda)\]
    such that the natural functor $\Shv_\mon(U\backslash G/U,\Lambda)\to \Tr(\Shv_\mon(U\backslash G/U,\Lambda),\sigma_*)$ is identified with the Deligne--Lusztig induction functor
    \[\Ch^\mon_{G,\phi}\coloneqq  q_*p^!\colon \Shv_\mon(U\backslash G/U,\Lambda)\to \Rep(G(\FF_q),\Lambda)\]
    in \eqref{eq: horocycle finite}.
\end{thm}

Let $\hat{G}$ be the dual group over $\Lambda$. Let $\hat{T}$ be the maximal torus of $\hat{G}$ and $\hat{C}=\hat{T}\git W$ be the Chevalley quotient. Let $[q]\colon \hat{G}\to\hat{G}$ be the $q$-power map, and let $[q]\colon\hat{C}\to \hat{C}$ be the induced map. Recall $\End_{G(\FF_q)}(\Gamma_\psi)\simeq \cO(\hat{C}^{[q]})$ by \cite[Theorem 1.4.2]{Eteve-Jordan}. We will define the cofree tilting sheaf $\Xi\in\Shv_\mon(U\backslash G/U,\Lambda)$ in Definition \ref{def: big tilting}. The main result of this appendix is the following theorem.

\begin{thm}\label{thm: Ch big tilting}
    There is a natural isomorphism
    \[
        \Ch^\mon_{G,\phi}(\Xi)\simeq \Gamma_\psi\otimes_{\cO(\hat{C})}\cO(\hat{T}).
    \]
\end{thm}

\subsection{The Whittaker category}
We define the category $\Shv_\mon(U\backslash G/(U^-,\psi))$ of $(U^-,\psi)$-equivariant sheaves as in \cite[\S4.2.4]{ZTame}. The monodromic Hecke category $\Shv_\mon(U\backslash G/U,\Lambda)$ acts on this category by left convolution. 

Let $\widetilde{U^-}\to U^-$ denote the $\FF_q$-cover defined by pulling back the Artin--Schreier cover of $\GG_a$ along $e\colon U^-\to\GG_a$. Then $\Shv_\mon(U\backslash G/(U^-,\psi),\Lambda)$ is a direct summand of $\Shv_\mon(U\backslash G/\widetilde{U^-},\Lambda)$ as in \cite[(4.34)]{ZTame}.
Consider the projection map
\[\widetilde\pr\colon U\backslash UTU^-/\widetilde{U^-}\to T\times U^-/\widetilde{U^-}\simeq T\times \BB\FF_p.\]
We define the functors
\[\Delta^\psi= j_!\widetilde\pr^!((-)\boxtimes \psi)\colon\Shv_\mon(T,\Lambda)\to \Shv_\mon(U\backslash G/(U^-,\psi),\Lambda)
\]
\[\nabla^\psi=j_*\widetilde\pr^!((-)\boxtimes \psi)\colon \Shv_\mon(T,\Lambda)\to \Shv_\mon(U\backslash G/(U^-,\psi),\Lambda) \] 
where $j\colon U\backslash UTU^-/\widetilde{U^-}\hookrightarrow U\backslash G/\widetilde{U^-}$ is the open embedding of the big cell.
By \cite[Lemma 4.55]{ZTame}, we have the following lemma.

\begin{lemma}\label{lemma-finite-Whittaker-equivalence}
    The functor
    \begin{equation}\label{eq: Delta psi}
        \nabla^\psi\simeq \Delta^\psi\colon \Shv_\mon(T,\Lambda)\xrightarrow{\sim}\Shv(U\backslash G/(U^-,\psi),\Lambda)
    \end{equation}
    is an equivalence of categories.
\end{lemma}
Let $\widetilde{\Ch}\in \Shv_\mon(T,\Lambda)$ denote the cofree tilting sheaf in \cite[(4.6)]{ZTame}. Denote $\widetilde{\Delta}^\psi\coloneqq \Delta^\psi(\widetilde{\Ch})$. Recall that there is an equivalence of symmetric monoidal categories
\begin{equation}\label{eq: Ch}
    \Ch\colon \Ind\Coh(R_{\hat{T}})\xrightarrow{\sim} \Shv_\mon(T,\Lambda)
\end{equation}
where $R_{\hat{T}}$ is the moduli space of strongly continuous homomorphisms $\widehat\ZZ^p(1)\to \hat{T}$ over $\Lambda$ by \cite[Proposition 4.32]{ZTame}. After fixing a topological generator $\tau\in \widehat\ZZ^p(1)$, we can identify $R_{\hat{T}}$ with the union of completions of $\hat{T}$ along $\Lambda$-points of prime-to-$p$ order. By \eqref{eq: Delta psi} and \eqref{eq: Ch}, there is an equivalence of categories
\begin{equation}\label{eq: Ch psi}
    \Ch^\psi\colon \Ind\Coh(R_{\hat{T}})\simeq \Shv_\mon(U\backslash G/(U^-,\psi),\Lambda).
\end{equation}
In particular, the category $\Shv_\mon(T,\Lambda)$, and hence $\Shv_\mon(U\backslash G/(U^-,\psi),\Lambda)$ are naturally $\cO(\hat{T})$-linear. By \cite[Lemma 3.5.9]{Eteve-Jordan}, the monoidal category $\Shv_\mon(U\backslash G/U,\Lambda)$ is naturally linear over $\cO(\hat{C})$, and $\Shv_\mon(U\backslash G/(U^-,\psi),\Lambda)$ is naturally a $\cO(\hat{C})$-linear module category over $\Shv_\mon(U\backslash G/U,\Lambda)$.

Define the functor
\[\Av_\psi\colon \Shv_\mon(U\backslash G/U,\Lambda)\to \Shv_\mon(U\backslash G/(U^-,\psi),\Lambda)\]
by $\Av_\psi(\cF)\coloneqq \cF\star \widetilde{\Delta}^\psi$. Let $\Av^\psi$ denote the right adjoint of $\Av_\psi$. 

\begin{defn}\label{def: big tilting}
    We define the \emph{big tilting sheaf} $\Xi\coloneqq \Av^\psi(\widetilde\Delta^\psi)\in \Shv_\mon(U\backslash G/U,\Lambda)$.
\end{defn}

By \cite[Lemma 4.4.1]{BY-Koszul-duality}\cite{DLYZ2025}, the object $\Xi$ is isomorphic to the cofree monodromic tilting sheaf $\mathrm{Til}_{\dot{w}_0}^\mon\in \Shv_\mon(U\backslash G/U ,\Lambda)$ associated with the longest element $w_0\in W$ in \cite[Proposition 4.50]{ZTame}.

\begin{lemma}\label{lemma: Xi vs Av}
    There is a natural isomorphism of functors
    \[
    (-)\star\Xi\simeq \Av^\psi\circ\Av_\psi\colon \Shv_\mon(U\backslash G/U,\Lambda)\to \Shv_\mon(U\backslash G/U,\Lambda)
    \]
\end{lemma}
\begin{proof}
    It is easy to see that $\Av_\psi$ preserves compact objects, and hence $\Av^\psi$ is continuous. As $\Shv_\mon(U\backslash G/U,\Lambda)$ is a semi-rigid monoidal category, the functor $\Av^\psi$ is also $\Shv_\mon(U\backslash G/U,\Lambda)$-linear. The lemma follows.
\end{proof}

\begin{lemma}\label{lemma: O(T) O(C) vs Av}
    The following diagram commutes:
    \[\begin{tikzcd}
        \Ind\Coh(R_{\hat{T}}) \ar[r,"-\otimes^L_{\cO(\hat{C})}\cO(\hat{T})"]\ar[d,"\Ch^\psi"swap,"\simeq"] & \Ind\Coh(R_{\hat{T}}) \ar[d,"\Ch^\psi","\simeq"swap] \\
        \Shv_\mon(U\backslash G/(U^-,\psi),\Lambda) \ar[r,"\Av_\psi\circ\Av^\psi"] & \Shv_\mon(U\backslash G/(U^-,\psi),\Lambda).
    \end{tikzcd}\]
\end{lemma}
\begin{proof}
    As in Lemma \ref{lemma: Xi vs Av}, the functor $\Av_\psi\circ\Av^\psi$ is $\Shv_\mon(U\backslash G/U,\Lambda)$-linear. Therefore it suffices to compute
    \[\Av_\psi \circ \Av^\psi(\widetilde{\Delta}^\psi)\simeq \Av_\psi(\Xi)\simeq \Xi\star\widetilde{\Delta}^\psi.\]
    It suffices to show $\Xi\star\widetilde{\Delta}^\psi\simeq \widetilde{\Delta}^\psi\otimes_{\cO(\hat{C})}^L\cO(\hat{T})$. This follows from \cite[Theorem 3.5.28]{Eteve-Jordan}. Indeed, the object $\Xi$ corresponds to the object $\omega_{R_{\hat{T}}\times_{R_{\hat{T}}\git W}R_{\hat{T}}}$ in the category of Soergel bimodules, and the action on $\omega_{R_{\hat{T}}}\in \Ind\Coh(R_{\hat{T}})$ is given by 
    \[(p_2)_*\delta^!(\omega_{R_{\hat{T}}\times_{R_{\hat{T}}\git W}R_{\hat{T}}}\boxtimes\omega_{R_{\hat{T}}})\simeq (p_2)_*\omega_{R_{\hat{T}}\times_{R_{\hat{T}}\git W}R_{\hat{T}}}\simeq\omega_{R_{\hat{T}}}\otimes^L_{\cO(\hat{C})}\cO(\hat{T})\]
    for
    \[
        (R_{\hat{T}}\times_{R_{\hat{T}}\git W} R_{\hat{T}})\times R_{\hat{T}}\xleftarrow{\delta}R_{\hat{T}}\times_{R_{\hat{T}}\git W} R_{\hat{T}}\xrightarrow{p_2} R_{\hat{T}}.
    \]
\end{proof}

\subsection{Proof of Theorem \ref{thm: Ch big tilting}}
We refer to \cite{GKRV-toy-model} and \cite{ZTame} for the definitions and properties of categorical traces.
Write $\bH=\Shv_\mon(U\backslash G/U,\Lambda)$ and $\bM=\Shv_\mon(U\backslash G/(U^-,\psi),\Lambda)$ for simplicity.
The $\bH$-module $\bM$ is left dualizable by \cite[Corollary 8.78]{ZTame}. Moreover, there is a natural $\bH$-linear functor
\[
    \can\colon \bM\to \bH^F\otimes_{\bH}\bM
\]
defined by $M\mapsto 1\otimes F^{-1}(M)$ for $F=\sigma_*$. Here $\bH^F$ is the $\bH$-bimodule with the right $\bH$-action given by $F$. By \cite[(7.61)]{ZTame}, this defines an object
\[
    \cl(\bM,\can)\in \Tr(\bH,F)=\Rep(G(\FF_q),\Lambda).
\]
By the argument in the proof of \cite[Theorem 4.135]{ZTame}, there is a canonical isomorphism
\begin{equation}
    \cl(\bM,\can)\simeq \Gamma_\psi.
\end{equation}
Consider the $\bH$-module map
\begin{equation}\label{eq: a Xi}
    a_\Xi\colon \bH\to \bH^F\otimes_\bH\bH\simeq\bH^F
\end{equation}
defined by $a_\Xi(X)=X\star\Xi$. We obtain the object
\begin{equation}\label{eq: Ch Xi as trace}
    \Ch^\mon_{G,\phi}(\Xi)\simeq \cl(\bH,a_\Xi)\in \Tr(\bH,F)=\Rep(G(\FF_q),\Lambda).
\end{equation}
We interpret categorical traces as traces in the Morita category $\Morita_\Lambda$ defined in \cite[\S3.6]{GKRV-toy-model}. Objects in $\Morita_\Lambda$ are identified with $\Alg(\PR^L_\Lambda)$, and morphisms in $\Morita_\Lambda$ are given by bimodules. By Lemma \ref{lemma: Xi vs Av}, we can write \eqref{eq: a Xi} as a composition of left $\bH$-module homomorphisms
\[
    \bH\xrightarrow{1\otimes\Av_\psi}\bH\otimes_\bH\bM \xrightarrow{1\otimes\Av^\psi\otimes 1} \bH\otimes_\bH\bH^F\otimes_\bH\bH.
\]
Therefore the object \eqref{eq: Ch Xi as trace} is induced by the 2-commutative diagram
\[\begin{tikzcd}[column sep=huge]
    \underline{D(\Lambda)} \ar[d,equal]\ar[r,"\bH"] & \underline{\bH} \ar[d,"\bH^F"] \\
    \underline{D(\Lambda)} \ar[d,equal]\ar[r,"\bM"]\ar[ru,Rightarrow,"\Av^\psi"] & \underline{\bH} \ar[d,equal] \\
    \underline{D(\Lambda)} \ar[r,"\bH"]\ar[ru,Rightarrow,"\Av_\psi"] & \underline{\bH}
\end{tikzcd}\]
in $\Morita_\Lambda$. By cyclicity of trace (\cite[(3.5)]{GKRV-toy-model}), we see that \eqref{eq: Ch Xi as trace} is also induced by the 2-commutative diagram
\[\begin{tikzcd}[column sep=huge]
    \underline{D(\Lambda)} \ar[d,equal]\ar[r,"\bM"] & \underline{\bH} \ar[d,equal] \\
    \underline{D(\Lambda)} \ar[d,equal]\ar[r,"\bH"]\ar[ru,Rightarrow,"\Av_\psi"] & \underline{\bH} \ar[d,"\bH^F"] \\
    \underline{D(\Lambda)} \ar[r,"\bM"]\ar[ru,Rightarrow,"\Av^\psi"] & \underline{\bH}.
\end{tikzcd}\]
The composition
\[
    \bM\xrightarrow{\Av^\psi\otimes 1}\bH^F\otimes_\bH\bH\xrightarrow{1\otimes1\otimes\Av_\psi}\bH^F\otimes_\bH\bH\otimes_\bH \bM\simeq \bH^F\otimes_\bH\bM
\]
is computed by
\[
    M\mapsto 1\otimes \Av_\psi\circ F^{-1}\circ \Av^\psi(M)\simeq\can(M)\otimes_{\cO(\hat{C})}^L\cO(\hat{T})
\]
by Lemma \ref{lemma: O(T) O(C) vs Av}. Therefore we have
\[
    \Ch^\mon_{G,\phi}(\Xi)\simeq \cl(\bM,\can\otimes^L_{\cO(\hat{C})}\cO(\hat{T}))\simeq \cl(\bM,\can)\otimes^L_{\cO(\hat{C})}\cO(\hat{T})\simeq \Gamma_\psi\otimes^L_{\cO(\hat{C})}\cO(\hat{T}).
\]
As $\cO(\hat{T})$ is finite free over $\cO(\hat{C})$ by \cite{Steinberg1975Pittie}, we have $\Gamma_\psi\otimes^L_{\cO(\hat{C})}\cO(\hat{T})\simeq \Gamma_\psi\otimes_{\cO(\hat{C})}\cO(\hat{T})$. This proves the theorem.

\bibliographystyle{amsalpha}
\bibliography{refs}
\end{document}